\documentclass[11pt,reqno]{amsart}
\usepackage{amsfonts}  
\usepackage{amssymb}  
 \usepackage{multicol}
 \numberwithin{equation}{section}

\newcommand{\R}{\mathbb R}  
\newcommand{\Z}{\mathbb Z}  
\newcommand{\C}{\mathbb C}  

\newcommand{\hf}{\frac{1}{2}}  
\newcommand{\intR}{\int_{-\infty}^{\infty}}
\renewcommand{\l}{\left}
\renewcommand{\r}{\right}

\renewcommand{\Re}{{\mathrm{Re\,}}}  
\renewcommand{\Im}{{\mathrm{Im\,}}}  

\newcommand{\sm}[4]{\l(\begin{smallmatrix} #1 & #2 \\ #3 & #4\end{smallmatrix}\r)}

\newcommand{\E}{{\mathcal E}}

\newcommand{\eps}{\varepsilon}

\newcommand{\mm}{\,{\mathrm{mod}}\,}

\newcommand{\Wig}{{\mathrm{Wig}}}
\newcommand{\Lan}{\,\big\langle}
\newcommand{\Ran}{\big\rangle\,}
\newcommand{\Fymp}{{\mathcal F}^{\mathrm{symp}}}

\newcommand{\Sqo}{{\mathrm{Sqf}}^{\mathrm{odd}}}

\newcommand{\notdiv}{\nmid}

\newcommand{\slut}{\end{document}}

\begin{document}  

\theoremstyle{plain}
 \newtheorem{theorem}{Theorem}[section]
 \newtheorem{proposition}[theorem]{Proposition}
 \newtheorem{lemma}[theorem]{Lemma}
 \newtheorem{corollary}[theorem]{Corollary}
 \theoremstyle{definition}
 \newtheorem{definition}[theorem]{Definition}


\title[The Riemann hypothesis]
{The Riemann hypothesis}

\author{Andr\'e Unterberger}
 \address{Math\'ematiques, CNRS UMR 9008, Universit\'e de Reims, BP 1039, B.P.1039\\ F-51687 Reims Cedex, France}

\email{andre.unterberger@univ-reims.fr, andre.unterberger@gmail.com}

\maketitle


{\sc{Abstract.}} The Weyl symbolic calculus of operators leads to the construction, if one takes for symbol a certain distribution decomposing over the set of zeros of the Riemann zeta function, of an operator with the following property: the Riemann hypothesis is equivalent to the validity of a collection of estimates involving this operator. Pseudodifferential arithmetic, a novel chapter of pseudodifferential operator theory, makes it possible to make the operator under study fully explicit. This leads to a proof of the conjecture.\\

\section{Introduction}

A discussion of the Riemann hypothesis cannot depend on complex analysis, in the style of the XIXth century, only. Though the analytic aspects of the Riemann zeta function must indeed contribute, arithmetic, especially prime numbers, must of necessity enter the discussion in a decisive role. The main part of the present chapter consists in putting together two aspects of the same object -- to wit, a certain parameter-dependent hermitian form -- one giving rise to a generating series with poles at the zeros of the zeta function, and the other described in terms of congruence algebra. A complete understanding of the latter one is reached with the help of pseudodifferential arithmetic, a novel chapter of pseudodifferential operator theory adapted to the species of symbols needed here.\\

The central object of the proof to follow is the distribution
\begin{equation}\label{int1}
{\mathfrak T}_{\infty}(x,\xi)=\sum_{|j|+|k|\neq 0} a^{\flat}((j,k))\,\delta(x-j)\,\delta(\xi-k)
\end{equation}
in the plane, with $(j,k)={\mathrm{g.c.d.}}(j,k)$ and $a^{\flat}(r)=\prod_{p|r}(1-p)$ for $r=1,2,\dots$. There is a collection $\l({\mathfrak E}_{\nu}\r)_{\nu\neq \pm 1}$ of so-called Eisenstein distributions, ${\mathfrak E}_{\nu}$ homogeneous of degree $-1-\nu$, such that, as an analytic functional,
\begin{equation}\label{12}
{\mathfrak T}_{\infty}=12\,\delta_0+\sum_{\zeta(\rho)=0}
{\mathrm{Res}}_{\nu=\rho}\l(\frac{{\mathfrak E}_{-\nu}}{\zeta(\nu)}\r),
\end{equation}
where $\delta_0$ is the unit mass at the origin of $\R^2$. This decomposition calls for using ${\mathfrak T}_{\infty}$ in a possible approach to the zeros of the Riemann zeta function: only, to obtain a full benefit of this distribution, one must appeal to pseudodifferential analysis, more precisely to the Weyl symbolic calculus of operators \cite{weyl}. This is the rule $\Psi$ that associates to any distribution ${\mathfrak S}\in {\mathcal S}'(\R^2)$ the so-called operator with symbol ${\mathfrak S}$, to wit the linear operator $\Psi({\mathfrak S})$ from ${\mathcal S}(\R)$ to ${\mathcal S}'(\R)$ weakly defined by the equation
\begin{equation}\label{013}
\l(\Psi({\mathfrak S})\,u\r)(x)=\hf\,\int_{\R^2} {\mathfrak S}\l(\frac{x+y}{2},\,\xi\r) e^{i\pi(x-y)\xi}\,u(y)\,dy\,d\xi.
\end{equation}
It defines pseudodifferential analysis, which has been for more than half a century one of the main tools in the study of partial differential equations. However,
the methods used in this regard in the present context do not intersect the usual ones (though familiarity with the more formal aspects of the Weyl calculus will certainly help) and may call for the denomination of ``pseudodifferential arithmetic'' (Sections 6 to 8). \\

Making use of the Euler operator $2i\pi\E=1+x\frac{\partial}{\partial x}+\xi\frac{\partial}{\partial \xi}$ and of the collection of rescaling operators $t^{2i\pi\E}$, to wit
\begin{equation}\label{14}
\l(t^{2i\pi\E}{\mathfrak S}\r)(x,\xi)=t\,{\mathfrak S}\l(tx,\,t\xi\r),\qquad t>0.
\end{equation}
we brought attention in \cite{birk18} to the hermitian form  $\l(w\,\bigr|\,\Psi\l(Q^{2i\pi\E}{\mathfrak T}_{\infty}\r) w\r)$. The following criterion for R.H. was obtained: that, for some $\beta>2$ and any function $w\in C^{\infty}(\R)$ supported in $[0,\beta]$, one should have for every $\eps>0$
\begin{equation}\label{15}
\l(w\,\bigr|\,\Psi\l(Q^{2i\pi\E}{\mathfrak T}_{\infty}\r) w\r)=
{\mathrm{O}}\l(Q^{\hf+\eps}\r)
\end{equation}
as $Q\to \infty$ through squarefree integral values. The proof is based on a spectral-theoretic analysis of the operator $\sum_{Q\in\Sqo} \Psi\l(Q^{-s+2i\pi\E}{\mathfrak S}\r)$, where $\Sqo$ denotes the set of squarefree odd integers, as a function of $s\in \C$: ultimately, if one assumes (\ref{15}), a pole with no right to exist will be associated to any zero $\rho$ of zeta such that $\Re\rho>\hf$. We shall prove here a slightly different criterion, involving a pair $v,u$ in place of $w,w$, in Section 5: the case when the supports of $v$ and $u$ do not intersect, nor do those of $v$ and $\overset{\vee}{u}$, is especially important.\\

Eisenstein distributions will be described in detail in Section 3. They make up just a part of the domain of automorphic distribution theory, which relates to the classical one of modular form theory but is more precise. The operator $\pi^2\E^2$ in the plane transfers under some map (the dual Radon transformation, in an arithmetic context) to the operator $\Delta-\frac{1}{4}$ in the hyperbolic half-plane, where $\Delta$ is the automorphic Laplacian. While this is crucial in other applications, it is another feature of automorphic distribution theory that will be essential here: the way it can cooperate with the Weyl symbolic calculus. Automorphic distribution theory has been developed in a series of books, ending with \cite{birk15}. Let us make it quite clear that the automorphy concepts will hardly be needed in this chapter, though both ${\mathfrak T}_{\infty}$ and the Eisenstein distributions owe their existence to such considerations. All these matters will be treated in Chapter 2.\\

The following property of the operator $\Psi({\mathfrak E}_{-\nu})$ with symbol ${\mathfrak E}_{-\nu}$ (shared by all operators with automorphic symbols) will be decisive in algebraic calculations: if $v,u$ is a pair of $C^{\infty}$ functions on the line such that $0<x^2-y^2<8$ whenever $\overline{v}(x)u(y)\neq 0$, the hermitian form $\l(v\,\big|\,\Psi({\mathfrak E}_{-\nu})\,u\r)$ depends only on the restriction of $\overline{v}(x)u(y)$ to the set $\{(x,y)\colon x^2-y^2=4\}$.\\

The first step towards a computation of the main hermitian form on the left-hand side of (\ref{15}) (or its polarization) consists in transforming it into a finite-dimensional hermitian form. Given a positive integer $N$, one sets
\begin{equation}\label{17}
{\mathfrak T}_N(x,\xi)=\sum_{j,k\in\Z} a^{\flat}((j,k,N))\,\delta(x-j)\,\delta(\xi-k).
\end{equation}
The distribution ${\mathfrak T}_{\infty}$ is obtained as the limit, as $N\nearrow \infty$ (a notation meant to express that $N$ will be divisible by any given squarefree integer, from a certain point on), of the distribution ${\mathfrak T}_N^{\times}$ obtained from ${\mathfrak T}_N$ by dropping the term corresponding to $j=k=0$. If the algebraic sum of supports of the functions $v,\,u\in C^{\infty}(\R)$ is contained in $[0,2\beta]$ and if $N=RQ$ is a squarefree integer divisible by all primes less than $\beta Q$, one has
\begin{equation}\label{18}
\l(v\,\bigr|\,\Psi\l(Q^{2i\pi\E}{\mathfrak T}_{\infty}\r) u\r)=\l(v\,\bigr|\,\Psi\l(Q^{2i\pi\E}{\mathfrak T}_N\r) u\r).
\end{equation}\\

Now, the hermitian form on the right-hand side is amenable to an algebraic-arithmetic version. Indeed, transferring functions in ${\mathcal S}(\R)$ to functions on $\Z/(2N^2)\Z$ under the linear map $\theta_N$ defined by the equation
\begin{equation}\label{19}
\l(\theta_Nu\r)(n)=\sum_{\begin{array}{c} n_1\in \Z \\ n_1\equiv n\mm 2N^2\end{array}} u\l(\frac{n_1}{N}\r),\qquad n\mm 2N^2,
\end{equation}
one obtains an identity
\begin{equation}\label{0110}
\l(v\,\bigr|\,\Psi\l(Q^{2i\pi\E}{\mathfrak T}_N\r) u\r)=\sum_{m,n\mm 2N^2}c_{R,Q}\l({\mathfrak T}_N;\,m,n\r)\,\,\overline{\theta_Nv}(m)\,\l(\theta_Nu\r)(n).
\end{equation}
The coefficients $c_{R,Q}\l({\mathfrak T}_N;\,m,n\r)$ are fully explicit, and the symmetric matrix defining this hermitian form has a Eulerian structure. The more important point is to separate, so to speak, the $R$-factor and $Q$-factor in this expression.\\
Introduce the reflection $n\mapsto \overset{\vee}{n}$ of $\Z/(2N^2)\Z$ such that $\overset{\vee}{n}\equiv n\mm R^2$ and $\overset{\vee}{n}\equiv -n\mm Q^2$. Then, constructing $\widetilde{u}$ so that $(\theta_N\widetilde{u})(n)=(\theta_Nu)(\overset{\vee}{n})$, one has the identity
\begin{equation}
\l(v\,\big|\,\Psi\l(Q^{2i\pi\E}{\mathfrak T}_N\r) u\r)=\mu(Q)\,\l(v\,\big|\,\Psi({\mathfrak T}_N)\,\widetilde{u}\r),
\end{equation}
which connects the purely arithmetic map $n\mapsto \overset{\vee}{n}$ to the purely analytic rescaling operator $Q^{2i\pi\E}$.\\

The contributions that precede, to be developed in full in Sections 1 to 6, were for the most already published in \cite{birk18}: specializing the parity of certain integers has led to a great simplification. It is at this point that, as a research paper, the present work really starts. Under some strong support conditions relative to $v,u$, one obtains the identity
\begin{align}
\l(v\,\big|\,\Psi(Q^{2i\pi\E}{\mathfrak T}_N)\,u\r)&=\sum_{Q_1Q_2=Q} \mu(Q_1)\nonumber\\
&\sum_{R_1|R} \mu(R_1)\,\overline{v}\l(\frac{R_1}{Q_2}+\frac{Q_2}{R_1}\r)\,
u\l(\frac{R_1}{Q_2}-\frac{Q_2}{R_1}\r),
\end{align}
in which $\mu$ is the M\"obius indicator.\\

In view of the role of this identity, we have given two quite different proofs of it.
Nevertheless, it suffers from the presence of coefficients $\mu(R_1)$, with $C_1Q<R_1< C_2Q$, too large a domain for the necessities. However, if one constrains the support of $u$ further, replacing for some $\eps>0$ $u$ by $u_Q$ such that $u_Q(y)=Q^{\frac{\eps}{2}}\,u(Q^{\eps}y)$, it follows from this identity that, if one assumes that $v$ is supported in $[2,\sqrt{8}]$ and that $u$ is supported in $[0,1]$, the function $F_{\eps}^{\,\flat}$ defined for $\Re s$ large as the sum of the series
\begin{equation}
F_{\eps}^{\,\flat}(s)=\sum_{Q\in \Sqo} Q^{-s}\l(v\,\big|\,\Psi\l(Q^{2i\pi\E}{\mathfrak T}_{\infty}\r) u_Q\r)
\end{equation}
is entire. This is a good substitute for our inability to prove the same about the function $F_0^{\,\flat}(s)$ (the same as $F_{\eps}^{\,\flat}(s)$ when $\eps$ is replaced by $0$): the function $F_{\eps}^{\,\flat}$ is much more tractable than $F_0^{\,\flat}$, and $F_{\eps}^{\,\flat}(s)$ goes to $F_0^{\,\flat}(s)$ as $\eps\to 0$ if $\Re s>2$. Note the importance of having introduced two independent rescaling operators, as made possible by the use of pseudodifferential analysis. But, doing so, we have destroyed the possibility to apply the criterion (\ref{15}), and we must reconsider this question.\\

There are now two decompositions into homogeneous parts (of distributions in the plane or on the line) to be considered. Writing $u=\frac{1}{i}\int_{\Re\mu=a} u^{\mu}d\mu$, where $u^{\mu}$ is homogeneous of degree $-\hf-\mu$, one obtains the formula
\begin{multline}
F_{\eps}^{\,\flat}(s)=-\frac{1}{i\pi}\int_{\Re\nu=c} \frac{\l(1-2^{-\nu}\r)^{-1}}{\zeta(\nu)}\,d\nu\\
\int_{\Re\mu=0} f(s-\nu+\eps\mu)\,\l(v\,\big|\,\Psi\l({\mathfrak E}_{-\nu}\r)u^{\mu}\r)d\mu,\qquad c>1,
\end{multline}
with
\begin{equation}
f(s)=\l(1+2^{-s}\r)^{-1}\,\times\,\frac{\zeta(s)}
{\zeta(2s)}.
\end{equation}\\

Three types of singularities are present: the factor $f(s-\nu+\eps\mu)$ is singular when $s-\nu+\eps\mu=1$, the factor $\zeta(\nu)$ appears in the denominator, and the function $u^{\mu}$ may have a pole at $\mu=-\hf$. In the strip $\{s\colon c+\hf<\Re(s-1)<c+1$, this double integral differs by an analytic term from the function $-i\,G_{\eps}^{\,\flat}(s)$, with
\begin{equation}\label{defGepsmu}
G_{\eps}^{\,\flat}(s)=\frac{4}{\pi^2}\,\int_{\Re \mu=0}\frac{\l(1-2^{-s+1-\eps\mu}\r)^{-1}}{\zeta(s-1+\eps\mu)}\,
\Phi(v,\,u;\,s-1+\eps\mu,\,\mu)\,d\mu
\end{equation}
and
\begin{equation}
\Phi(v,\,u;\,s-1+\eps\mu,\,\mu)=\int_0^{\infty} \overline{v}(t+t^{-1})\,t^{s-1+\eps\mu}\,u^{\mu}(t+t^{-1})\,\frac{dt}{t},
\end{equation}
in which only a $d\mu$-integral remains.\\

To analyze the function $G_{\eps}^{\,\flat}(s)$, we decompose the function $u^{\mu}$ into an ``ingoing'' and an outgoing parts, obtaining a decomposition $G_{\eps}^{\,\flat}(s)=G_{\eps}^{+,n}(s)+G_{\eps}^{-,n}(s)$, with two terms to be treated differently. Analyzing the residue at the crucial pole at $\mu=-\hf$, active in the second term only of this decomposition, will lead to the desired proof.\\

\noindent
{\em Remarks\/} 1.1. (i) Most of what precedes had been published in this series of arXiv preprints for some time already, for instance in v12, three years ago. In this document, we made at some point a wrong turn, embarking in an attempt at disproving R.H. Our reason for doing so is that we were interested in the Lindel\"of hypothesis as well, and the (wrong) proof we had in mind would at the same time have disproved R.H. Of course, a proof of R.H. implies that the Lindel\"of hypothesis holds as well (Hadamard). We were wrong. This is a pity, since we were very close to a correct conclusion, but we did not rush it as we did not expect a frantic competition with anyone: we believed, and still do, that pseudodifferential arithmetic is a must of the proof. The 7/8 result of AI incited us to review with extreme care our last attempt, and we sure did find an error, explained as our reason for withdrawing the arXiv preprint arXiv:2208.12937. We now come back to it: the new developments concern only the proof of the final Theorem \ref{theoRH}..\\

(ii) All analysts know how life would be simpler (and deprived of most of its mathematical interest) if it were not for the uncertainty principle, felt in various possible forms. For instance, a homogeneous function on the positive half-line can of course not be compactly supported: using the result of testing an operator (say, the one the symbol of which is an Eisenstein distribution, as in (\ref{910})) on a pair $v,u$ is one way to approach the difficulty: $u^{\mu}(y)$ is homogeneous, $v(x)$ is compactly supported and the integrand can be nonzero at some point only if $x^2=4+y^2$. A quite subtle other way to finesse the uncertainty principle is to decompose $u(y)$ as $u_+^{\mu,n}(y)+u_-^{\mu,n}(y)$, as done in Section 1.9. The two terms are almost homogeneous but look, so to speak, in the direction of increasing or decreasing $y$.\\

(iii) The reader may find the following hints regarding what is common in the present proof and that of the 7/8 AI result of interest. The starting points (Theorem \ref{theo53} or \cite[section\,3.4]{birk18} here and Proposition 2.1 in
the AI document of Sept.26, 2026) have some similarity going beyond the obvious fact that, to approach R.H., one should look at an object with the zeta function in the denominator of its decomposition of sorts. We have not had more than a cursory look at the AI document, but we feel that, combining into a simple object two objects built with sums associated to characters mod 2 and 3, is not completely foreign to the idea, expressed above, of piecing two functions and making a single object from the pair. Only, the composite object is a Wigner function (the universal sesquilinear object on a pair $v,u$) in our case, and still a function on the line in the AI document: this accounts for the necessity to make it quite specialized.\\

We would not be honest if we did not tell the influence AI had on our work: under the pressure of their result, we hurried to come back to the search for a proof of R.H., rather than a disproof.\\

\section{Some distributions of arithmetic interest}

In the plane, we make us of the Schwartz space ${\mathcal S}(\R^2)$ of rapidly decreasing functions of $(x,\xi)\in \R^2$, provided with its usual topology, and of the space ${\mathcal S}'(\R^2)$ of tempered distributions, to wit continuous linear forms on ${\mathcal S}(\R^2)$. One has ${\mathcal S}(\R^2)\subset L^2(\R^2)\subset {\mathcal S}'(\R^2)$. The group $G=SL(2,\R)$ acts on ${\mathcal S}(\R^2)$ by the map $(h,\,g)\mapsto h\,\circ\,g^{-1}$, and on ${\mathcal S}'(\R^2)$ by the map
$({\mathfrak S},\,g)\mapsto {\mathfrak S}\,\circ\,g^{-1}$, such that $\Lan {\mathfrak S}\,\circ\,g^{-1},\,h\Ran=\Lan {\mathfrak S},\,h\circ\,g\Ran$.\\

We shall interest ourselves mostly in automorphic distributions, by which we mean tempered distributions ${\mathfrak S}$ such that ${\mathfrak S}\,\circ\,g^{-1}={\mathfrak S}$ for $g\in \Gamma = SL(2,\Z)$. Such are the discrete measures
\begin{equation}
{\mathfrak S}_c(x,\xi)=c(0)\,\,\delta(x)\,\delta(\xi)+\sum_{\begin{array}{c} j,k\in \Z\\ |j++|k|\neq 0\end{array}} c((j,k))\,\delta(x-j)\,\delta(\xi-k),
\end{equation}
where $(j,k)=g.c.d.(j,\,k)$, a convention to follow, and $c(r)$ is any function of $r=0,1,\dots$ bounded by some power of $2+|r|$.\\

Denote as $\mu$ the M\"obius indicator function such that $\mu(r)=0$ unless $\mu$ is squarefree, in which case $\mu(r)=1$ or $-1$ according to the parity of the number of prime factors of $r$. Of special interest to us (and to zeta) are the cases when $a^{\flat}(r)=\mu(r)\,r$, and $a(r)=|a^{\flat}(r)|$. If one denotes, classically, as $\prod_p$ a product extended to the set of prime numbers, one has for $r=1,2,\dots$
\begin{equation}\label{A32}
a^{\flat}(r)=\prod_{p|r}(1-p),\qquad a(r)=\prod_{p|r}(1+p).
\end{equation}
The number $a(r)$, or $a^{\flat}(r)$, depends only on the ``squarefree version'' $r_{\bullet}$ of $r$ defined as the product of all prime factors of $r$.\\

Given $N=1,2,\dots$, the distributions
\begin{align}\label{A33}
{\mathfrak T}_N(x,\xi)&=\sum_{j,k\,\in\Z} a^{\flat}((j,k,N))\,\delta(x-j)\,\delta(\xi-k),\nonumber\\ {\mathfrak S}_N(x,\xi)&=\sum_{j,k\,\in\Z} a((j,k,N))\,\delta(x-j)\,\delta(\xi-k)
\end{align}
depend only on $N_{\bullet}$. We also denote as ${\mathfrak T}_N^{\times}$ and
${\mathfrak S}_N^{\times}$the distributions obtained from ${\mathfrak T}_N$ and ${\mathfrak S}_N$
by discarding the term $a^{\flat}(N)\,\delta(x)\delta(\xi)$, in other words by limiting the summation to all pairs of integers $j,k$ such that $|j|+|k|\neq 0$.\\

The symbol $N\nearrow \infty$ will used throughout the book to indicate that the positive integer $N$ goes to $\infty$ through a sequence such that every given squarefree number $M=1,2,\dots$ divides $N$ for $N$ large enough: in particular, $(j,k,N)\to (j,k)$ as $N\nearrow \infty$. We use on the space ${\mathcal S}'(\R^2)$ its weak topology of dual of ${\mathcal S}(\R^2)$: it suffices here to say that the convergence of a sequence in this space is equivalent to the convergence, for every $h\in {\mathcal S}(\R^2)$, of the results of testing on $h$ the elements of that sequence.
We set
\begin{align}\label{A34}
{\mathfrak T}_{\infty}&={\mathrm{lim}}_{N\nearrow \infty}{\mathfrak T}_N^{\times}=
\sum_{\begin{array}{c} j,k\in \Z \\ |j|+|k|\neq 0\end{array}}a^{\flat}((j,k))\,\delta(x-j)\,\delta(\xi-k) ,\nonumber\\
{\mathfrak S}_{\infty}&={\mathrm{lim}}_{N\nearrow \infty}{\mathfrak S}_N^{\times}=
\sum_{\begin{array}{c} j,k\in \Z \\ |j|+|k|\neq 0\end{array}}a((j,k))\,\delta(x-j)\,\delta(\xi-k) .
\end{align}\\

The concept of Eisenstein distribution, to be introduced now, is the fundamental tool in the present book.\\

\begin{definition}\label{def35}
If $\nu\in\C,\ \Re\nu>1$, the Eisenstein distribution ${\mathfrak E}_{-\nu}$ is defined
by the equation, valid for every $h\in {\mathcal S}(\R^2)$,
\begin{equation}\label{A35}
\langle\,{\mathfrak E}_{-\nu}\,,\,h\,\rangle =\sum_{|j|+|k|\neq 0}
\int_0^{\infty}t^{\nu}h(jt,\,kt)\,dt.
\end{equation}\\
\end{definition}

It is immediate that the series of integrals converges if $\Re \nu>1$, in which case ${\mathfrak E}_{-\nu}$ is well defined as a tempered distribution. Indeed, writing $|h(x,\xi)|\leq C\,(1+|x|+|\xi|)^{-A}$ with $A$ large, one has
\begin{equation}\label{A36}
\l|\int_0^{\infty}t^{\nu}h(jt,kt)\,dt\r|\leq C\,(|j|+|k|)^{-\Re\nu-1}\int_0^{\infty}t^{\Re\nu}(1+t)^{-A}dt.
\end{equation}
Obviously, ${\mathfrak E}_{-\nu}$ is $SL(2,\Z)$--invariant as a distribution, i.e., an automorphic distribution. It is homogeneous of degree $-1+\nu$, i.e., $(2i\pi\E)\,{\mathfrak E}_{-\nu}=\nu\,{\mathfrak E}_{-\nu}$: note that the transpose of $2i\pi\E$ is $-2i\pi\E$. Its name stems from its relation (not needed here: the whole developments in this chapter are made in the plane, and references to the hyperbolic half-plane would only make things more obscure) with the classical notion of non-holomorphic Eisenstein series, as made explicit in \cite[p.93]{birk18}: it is, however, a more precise concept.\\

\begin{proposition}\label{prop32}
As a tempered distribution, \ ${\mathfrak E}_{-\nu}$ extends as a meromorphic function
of $\nu\in \C$, whose only poles are the points $\nu=\pm 1$: these poles are
simple, and the residues of  \ ${\mathfrak E}_{\nu}$ there are
\begin{equation}\label{A37}
{\mathrm{Res}}_{\nu=1}\,{\mathfrak E}_{-\nu}=1\qquad {\mathrm{and}}\qquad
{\mathrm{Res}}_{\nu=-1}\,{\mathfrak E}_{-\nu}=-\delta_0,
\end{equation}
the negative of the unit mass at the origin of $\R^2$. Defining the symplectic Fourier transformation $\Fymp$ by the equation $(\Fymp {\mathfrak S})(x,\xi)\\
= \int_{\R^2}{\mathfrak S}(y,\eta)\,e^{2i\pi(x\eta-y\xi)}\,dy\,d\eta$, one has, for $\nu\neq \pm 1$, $\Fymp{\mathfrak E}_{-\nu}={\mathfrak E}_{\nu}$.\\
\end{proposition}

\begin{proof}
Denote as $\l({\mathfrak E}_{-\nu}\r)_{\mathrm{princ}}$ ({\em resp.\/}
$\l({\mathfrak E}_{-\nu}\r)_{\mathrm{res}}$) the distribution defined in the same way as ${\mathfrak E}_{-\nu}$, except for the fact that in the integral (\ref{A35}), the interval of integration $(0,\infty)$ is replaced by the interval $(0,1)$ (resp.\, $(1,\infty)$), and observe that the distribution $\l({\mathfrak E}_{-\nu}\r)_{\mathrm{res}}$ extends as an entire function of $\nu$\,: indeed, it suffices to replace (\ref{A36}) by the inequality
$(1+(|j|+|k|)\,t)^{-A}\leq C\,(1+(|j|+|k|))^{-A}(1+t)^{-A}$ if $t>1$. As a consequence of Poisson's formula, one has when  $\Re\nu>1$ the identity
\begin{multline}
\int_1^{\infty} t^{-\nu}\sum_{(j,k)\in \Z^2} \l({\mathcal F}^{\mathrm{symp}}h\r)(tk,\,tj)\,dt=
\int_1^{\infty} t^{-\nu}\sum_{(j,k)\in \Z^2} t^{-2}\,h(t^{-1}j,\,t^{-1}k)\,dt\\
=\int_0^1 t^{\nu}\sum_{(j,k)\in \Z^2}\, h(tj,\,tk)\,dt,
\end{multline}
from which one obtains that
\begin{equation}
<\,{\mathcal F}^{\mathrm{symp}}\,\l({\mathfrak E}_{\nu}\r)_{\mathrm{res}}\,,\,h>-\frac{({\mathcal
F}^{\mathrm{symp}}h)(0,\,0)}{1-\nu}=<\,\l({\mathfrak E}_{-\nu}\r)_{\mathrm{princ}}\,,\,h>+\frac{h(0,\,0)}{1+\nu}.
\end{equation}
From this identity, one finds the meromorphic continuation of the function $\nu\mapsto
{\mathfrak E}_{\nu}$, including the residues at the two poles, as well as the fact that
${\mathfrak E}_{\nu}$ and ${\mathfrak E}_{-\nu}$ are the images of each other under
${\mathcal F}^{\mathrm{symp}}$.\\
\end{proof}

So far as the zeta function is concerned, we recall its definition $\zeta(s)=\sum_{n\geq 1}n^{-s}=\prod_p(1-p^{-s})^{-1}$, valid for $\Re s>1$, and the fact that it extends as a meromorphic function in the entire complex plane, with a single simple pole, of residue $1$, at $s=1$; also that, with $\zeta^*(s)=\pi^{-\frac{s}{2}}\Gamma(\frac{s}{2})\,\zeta(s)$, one has the functional equation $\zeta^*(s)=\zeta^*(1-s)$.\\

\begin{lemma}\label{lem33}
One has if $\nu\neq \pm 1$ the Fourier expansion
\begin{equation}\label{37}
{\mathfrak E}_{-\nu}(x,\xi)=\zeta(\nu)\,|\xi|^{\nu-1}
+\zeta(1+\nu)\,|x|^{\nu}\delta(\xi)+\sum_{r\neq 0} \sigma_{-\nu}(r)\,|\xi|^{\nu-1}
\exp\l(2i\pi\frac{rx}{\xi}\r)
\end{equation}
where $\sigma_{-\nu}(r)=\sum_{1\leq d|r  } d^{-\nu}$: the first two terms must be grouped when $\nu=0$. Given $b>\eps>0$, the distribution $(\nu-1){\mathfrak E}_{-\nu}$ remains in a bounded subset of ${\mathcal S}'(\R^2)$ for $\eps\leq \Re \nu\leq b$.\\
\end{lemma}

\begin{proof}
Isolating the term for which $k=0$ in (\ref{A35}), we write if $\Re\nu>1$, after a change of variable,
\begin{multline}
{}\Lan {\mathfrak E}_{-\nu},\,h\Ran=\zeta(1+\nu)\intR |t|^{\nu}h(t,0)\,dt+\hf \sum_{j\in\Z,\,k\neq 0} \intR |t|^{\nu}h(jt,\,kt)\,dt\\
=\zeta(1+\nu)\intR |t|^{\nu}h(t,0)\,dt
+\hf\sum_{j\in\Z,\,k\neq 0} \intR |t|^{\nu-1}\l({\mathcal F}_1^{-1}h\r)\l(\frac{j}{t},\,kt\r) dt,
\end{multline}
where we have used Poisson's formula at the end and denoted as ${\mathcal F}_1^{-1}h$
the inverse Fourier transform of $h$ with respect to the first variable. Isolating now the term such that $j=0$, we obtain
\begin{multline}
\hf\sum_{j\in\Z,\,k\neq 0} \intR |t|^{\nu-1}\l({\mathcal F}_1^{-1}h\r)\l(\frac{j}{t},\,kt\r) dt=\zeta(\nu)\intR |t|^{\nu-1}\l({\mathcal F}_1^{-1}h\r)(0,\,t)\,dt\\
+\hf \sum_{jk\neq 0}\intR |t|^{\nu-1}\l({\mathcal F}_1^{-1}h\r)\l(\frac{j}{t},\,kt\r) dt,
\end{multline}
from which the main part of the lemma follows if $\Re\nu>1$ after we have made the change of variable $t\mapsto \frac{t}{k}$ in the main term. The continuation of the identity uses also the fact that,
thanks to the trivial zeros $-2,-4,\dots$ of zeta, the product $\zeta(\nu)\,|t|^{\nu-1}$ is regular at $\nu=-2,-4,\dots$ and the product $\zeta(1+\nu)\,|t|^{\nu}$ is regular at $\nu=-3,-5,\dots$, . That the sum $\zeta(\nu)\,|\xi|^{\nu-1}+\zeta(1+\nu)\,|x|^{\nu}\delta(\xi)$ is regular at $\nu=0$ follows from the facts that $\zeta(0)=-\hf$ and that the residue at $\nu=0$ of the distribution $|\xi|^{\nu-1}=\frac{1}{\nu}\,\frac{d}{d\xi}\l(|x|^{\nu}{\mathrm{sign}}\,\xi\r)$ is
$\frac{d}{d\xi}\,{\mathrm{sign}}\,\xi=2\delta(\xi)$.\\

The second assertion is a consequence of the Fourier expansion. The factor $(\nu-1)$ has been inserted so as to kill the pole of ${\mathfrak E}_{-\nu}$, or that of $\zeta(\nu)$, at $\nu=1$. Bounds for the first two terms of the right-hand side of (\ref{37}) are obtained from rough bounds for the zeta factors and integrations by parts associated to powers of the operator $\xi\frac{\partial}{\partial \xi}$ or $x\frac{\partial}{\partial x}$. For the main terms, we use the integration by parts associated to the equation
\begin{equation}
\l(1+\xi\frac{\partial}{\partial x}\r)\exp\l(2i\pi\frac{rx}{\xi}\r)=(1+2i\pi r)\,
\exp\l(2i\pi\frac{rx}{\xi}\r).
\end{equation}\\
\end{proof}

Decompositions into homogeneous components of functions or distributions in the plane will be ever-present. Any function $h\in {\mathcal S}(\R^2)$ can be decomposed in $\R^2\backslash\{0\}$ into homogeneous components according to the equations, in which $c>-1$,
\begin{equation}\label{310}
h=\frac{1}{i}\int_{\Re\nu=c}h_{\nu}\,d\nu,\qquad h_{\nu}(x,\xi)=\frac{1}{2\pi}\int_0^{\infty}t^{\nu}h(tx,t\xi)\,dt.
\end{equation}
Indeed, the integral defining $h_{\nu}(x,\xi)$ is convergent for $|x|+|\xi|\neq 0,\Re\nu>-1$, and the function $h_{\nu}$ so defined is $C^{\infty}$ in $\R^2\backslash\{0\}$ and homogeneous of degree $-1-\nu$; it is also analytic with respect to $\nu$. Using twice the integration by parts associated to Euler's equation $-(1+\nu)\,h_{\nu}=\l(x\frac{\partial}{\partial x}+\xi\frac{\partial}{\partial \xi}\r)h_{\nu}(x,\xi)$, one sees that the integral $\frac{1}{i}\int_{\Re\nu=c}h_{\nu}(x,\xi)\,d\nu$ is convergent for $c>-1$: its value does not depend on $c$. Taking $c=0$ and setting $t=e^{2\pi \tau}$, one has for $|x|+|\xi|\neq 0$
\begin{equation}
h_{i\lambda}(x,\xi)=\intR e^{2i\pi \tau\lambda} \,.\,\,e^{2\pi \tau}h(e^{2\pi\tau}x,e^{2\pi\tau}\xi)\,d\tau,
\end{equation}
and the Fourier inversion formula shows that $\intR h_{i\lambda}(x,\xi)\,d\lambda=h(x,\xi)$: this proves (\ref{310}).\\

As a consequence, some automorphic distributions of interest (not all: so-called Hecke distributions are needed too in general, as will be explained in Chapter 2) can be decomposed into Eisenstein distributions. A basic one is the  ``Dirac comb''
\begin{equation}\label{312}
{\mathfrak D}(x,\xi)=2\pi \sum_{|j|+|k|\neq 0} \delta(x-j)\,\delta(\xi-k)
=2\pi\,\l[{\mathcal Dir}(x,\xi)-\delta(x)\delta(\xi)\r]
\end{equation}
where, as found convenient in some algebraic calculations, one introduces also the ``complete'' Dirac comb \,${\mathcal Dir}(x,\xi)=\sum_{j,k\in \Z}\delta(x-j)\delta(\xi-k)$.\\

Noting the inequality $\l|\int_0^{\infty} t^{\nu}h(tx,t\xi)\,dt\r|\leq C\,(|x|+|\xi|)^{-\Re\nu-1}$, one obtains if $h\in{\mathcal S}(\R^2)$ and $c>1$, pairing (\ref{312}) with (\ref{310}), the identity
\begin{equation}
{}\Lan {\mathfrak D},\,h\Ran=\frac{1}{i}\sum_{|j|+|k|\neq 0} \int_{\Re\nu=c} d\nu\int_0^{\infty}t^{\nu}h(tj,tk)\,dt.
\end{equation}
It follows from (\ref{A35}) that, for $c>1$,
\begin{equation}\label{314}
{\mathfrak D}=\frac{1}{i}\int_{\Re\nu=c} {\mathfrak E}_{-\nu}\,d\nu=2\pi+\frac{1}{i}\int_{\Re\nu=0} {\mathfrak E}_{-\nu}\,d\nu,
\end{equation}
the second equation being a consequence of the first in view of (\ref{A37}).\\

Integral superpositions of Eisenstein distributions, such as the one in (\ref{314}), are to be interpreted in the weak sense in ${\mathcal S}'(\R^2)$, i.e., they make sense when tested on arbitrary functions in ${\mathcal S}(\R^2)$. Of course, pole-chasing is essential when changing contours of integration. But, as long as one satisfies oneself with weak decompositions in ${\mathcal S}'(\R^2)$, no difficulty concerning the integrability with respect to $\Im \nu$ on the line ever occurs, because of the last assertion of Lemma \ref{lem33} and of the identities
\begin{equation}\label{A315}
(b-\nu)^A<{\mathfrak E}_{-\nu},W>=<(b-2i\pi\E)^A{\mathfrak E}_{-\nu},W>=
<{\mathfrak E}_{-\nu},\,(b+2i\pi\E)^AW>,
\end{equation}
in which $A=0,1,\dots$ may be chosen arbitrarily large and $b$ is arbitrary.\\

\begin{theorem}
For any squarefree integer $N\geq 1$, defining
\begin{equation}\label{43}
\zeta_N(s)=\prod_{p|N}(1-p^{-s})^{-1},\qquad {\mathrm{so\,\,that\,\,\,}} \frac{1}{\zeta_N(s)}=\sum_{1\leq T|N}\mu(T)\,T^{-s},
\end{equation}
one has
\begin{equation}\label{44}
{\mathfrak T}_N^{\times}=\frac{1}{2i\pi}\int_{\Re \nu=c}
\frac{1}{\zeta_N(\nu)}\,{\mathfrak E}_{-\nu}\,d\nu,\qquad c>1.
\end{equation}
Taking the weak limit as $N\nearrow \infty$, one has
\begin{equation}\label{45}
{\mathfrak T}_{\infty}=\frac{1}{2i\pi}\int_{\Re \nu=c} \frac{{\mathfrak
E}_{-\nu}}{\zeta(\nu)}\,d\nu,\qquad c\geq 1.
\end{equation}
On the other hand,
\begin{equation}
{\mathfrak S}_{\infty}=\frac{1}{2i\pi}\int_{\Re \nu=c} \frac{\zeta(\nu)}{\zeta(2\nu)}\,{\mathfrak
E}_{-\nu}\,d\nu,\qquad c> 1.
\end{equation}\\
\end{theorem}

\begin{proof}
Using the equation $T^{-x\frac{d}{dx}}\delta(x-j)=\delta\l(\frac{x}{T}-j\r)
=T\,\delta(x-Tj)$, one has with ${\mathfrak D}$ as introduced in (\ref{312})
\begin{align}
\frac{1}{2\pi} \prod_{p|N}\l(1-p^{-2i\pi\E}\r){\mathfrak
D}(x,\xi)&=\sum_{T|N}\mu(T)\,T^{-2i\pi\E}\sum_{|j|+|k|\neq 0} \delta(x-j)\,\delta(\xi-k)\nonumber\\
&=\sum_{T|N}\mu(T)\,T\,\sum_{|j|+|k|\neq 0}\delta(x-Tj)\,\delta(\xi-Tk)\nonumber\\
&=\sum_{T|N}\mu(T)\,T\,\sum_{\begin{array}{c}|j|+|k|\neq 0 \\ j\equiv k \equiv 0\mm T\end{array}}
\delta(x-j)\,\delta(\xi-k)\nonumber\\
&=\sum_{\begin{array}{c}|j|+|k|\neq 0 \\T|(N,j,k)\end{array}}
\mu(T)\,T\,\delta(x-j)\,\delta(\xi-k).
\end{align}
Since $\sum_{T|(N,j,k)}\mu(T)\,T=\prod_{p|(N,j,k)}\l(1-p\r)=
a^{\flat}((N,j,k))$, one obtains
\begin{equation}\label{47}
\frac{1}{2\pi} \prod_{p|N}\l(1-p^{-2i\pi\E}\r){\mathfrak D}(x,\xi)={\mathfrak
T}_N^{\times}(x,\xi).
\end{equation}
One has
\begin{multline}
\l[{\mathfrak T}_N-{\mathfrak T}_N^{\times}\r](x,\xi)=a^{\flat}(N)\,\delta(x)\delta(\xi)\\
=\delta(x)\delta(\xi)\,\prod_{p|N}(1-p)=
\prod_{p|N}\l(1-p^{-2i\pi\E}\r)\l(\delta(x)\delta(\xi)\r),
\end{multline}
so that, adding the last two equations,
\begin{equation}\label{49}
{\mathfrak T}_N=\prod_{p|N}\l(1-p^{-2i\pi\E}\r)\,{\mathcal Dir}.
\end{equation}\\

Combining (\ref{47}) with (\ref{314}) and with $(2i\pi\E){\mathfrak E}_{-\nu}=\nu\,{\mathfrak E}_{-\nu}$, one obtains if $c>1$
\begin{align}
{\mathfrak T}_N^{\times}&= \prod_{p|N}\l(1-p^{-2i\pi\E}\r)
\l[\frac{1}{2i\pi}\int_{\Re\nu=c} {\mathfrak E}_{-\nu}\,d\nu\r]\nonumber\\
&= \frac{1}{2i\pi}\int_{\Re\nu=c} {\mathfrak E}_{-\nu}\,
\prod_{p|N}\l(1-p^{-\nu}\r)\,d\nu,
\end{align}
which is just (\ref{44}). The $d\nu$-summability is guaranteed by (\ref{A315}).
Equation (\ref{45}) follows as well, taking the limit as $N\nearrow \infty$. Recall also (\ref{A315}).\\

In view of Hadamard's theorem, according to which $\zeta(s)$ has no zero on the line $\Re s=1$, to be completed by the estimate \cite[p.211]{ten} $(\zeta(1+iy))^{\pm 1}={\mathrm{O}}(\log\,|y|)$, observing also that the pole of $\zeta(\nu)$ at $\nu=1$ kills that of ${\mathfrak E}_{-\nu}$ there, one can in (\ref{45}) replace the condition $c>1$ by $c\geq 1$.\\

So far as ${\mathfrak S}_N^{\times}$ or ${\mathfrak S}_{\infty}$ is concerned, it suffices to replace (\ref{43}) by
\begin{equation}\label{410}
\prod_{p|N}(1+p^{-s})^{-1}=\prod_{p|N} \frac{\l(1-p^{-s}\r)^{-1}}{\l(1-p^{-2s}\r)^{-1}}=\frac{\zeta_N(s)}{\zeta_N(2s)}.
\end{equation}
\end{proof}

\noindent
{\em Remark\/} 2.1. In (\ref{45}), introducing a sum of residues over all zeros of zeta with a real part above some large negative number, one can replace the line $\Re\nu=c$ with $c>1$ by a line $\Re\nu=c'$ with $c'$ as negative as desired: one cannot go further in the distribution sense. But \cite[Theor.\,3.2.2,\,Theor.\,3.2.4]{birk18}, one can get rid of the integral if one agrees to interpret the identity in the sense of a certain analytic functional. Then, all zeros of zeta, non-trivial and trivial alike, enter the formula:
the ``trivial'' part
\begin{equation}
{\mathfrak R}_{\infty}=2\sum_{n\geq 0} \frac{(-1)^{n+1}}{(n+1)\,!}\,\frac{\pi^{\frac{5}{2}+2n}}{\Gamma(\frac{3}{2}+n)\zeta(3+2n)}\,
{\mathfrak E}_{2n+2}
\end{equation}
has a closed expression \cite[p.22]{birk18} as a series of line measures. This will not be needed in the sequel.\\

For reasons having to do with the Weyl symbolic calculus of operators, to be introduced in the next
section, we shall use preferably the distributions
\begin{align}\label{415}
{\mathfrak T}_{\frac{\infty}{2}}(x,\xi)&=\sum_{|j|+|k|\neq 0} a^{\flat}\big(\big(j,k,\frac{\infty}{2}\big)\big)\,\delta(x-j)\delta(\xi-k),\nonumber\\
{\mathfrak S}_{\frac{\infty}{2}}(x,\xi)&=\sum_{|j|+|k|\neq 0} a\big(\big(j,k,\frac{\infty}{2}\big)\big)\,\delta(x-j)\delta(\xi-k),
\end{align}
where $a\big(\big(j,k,\frac{\infty}{2}\big)\big)$ \,(resp.\,$a\big(\big(j,k,\frac{\infty}{2}\big)\big)$) is the product of all factors $1-p$ (resp.\,$1+p$)\,, with $p$ prime $\neq 2$ dividing $(j,k)$. Since $\prod_{p\neq 2}(1-p^{-\nu})=
\frac{(1-2^{-\nu})^{1}}{\zeta(\nu)}$ and $\prod_{p\neq 2}(1+p^{-\nu})=\frac{(1+2^{-\nu})^{1}}\,
\frac{\zeta(\nu}{\zeta(2\nu)}$, one has for $c>1$
\begin{align}\label{416}
{\mathfrak T}_{\frac{\infty}{2}}&=\frac{1}{2i\pi}\int_{\Re\nu=c}
\l(1-2^{-\nu}\r)^{-1}\frac{{\mathfrak E}_{-\nu}}{\zeta(\nu)}\,d\nu,\nonumber\\
{\mathfrak S}_{\frac{\infty}{2}}&=\frac{1}{2i\pi}\int_{\Re\nu=c}
\l(1+2^{-\nu}\r)^{-1}{\mathfrak E}_{-\nu}\,\frac{\zeta(\nu)}{\zeta(2\nu)}\,d\nu.
\end{align}\\

The version of (\ref{413}) below we shall use, under the same support conditions about $v,u$, is
\begin{equation}\label{417}
\l(v\,\big|\,\Psi\l(Q^{2i\pi\E}{\mathfrak T}_{\frac{\infty}{2}}\r)u\r)=
\l(v\,\big|\,\Psi\l(Q^{2i\pi\E}{\mathfrak T}_N\r)u\r)
\end{equation}
if $N$ and $Q$ are squarefree odd and $N$ is divisible by all odd primes $<\beta Q$.\\

\section{The Weyl calculus of operators and pseudodifferential arithmetic}

In space-momentum coordinates, the Weyl calculus, or pseudodifferential calculus, depends on one free parameter $h$ with the dimension of action, called Planck's constant. The traditional definition of the operator $\Psi_h({\mathfrak S})$ associated to a function, or distribution ${\mathfrak S}$ in the plane, is by means of the equation
\begin{equation}\label{21}
(\Psi_h({\mathfrak S})\,u)(x)=\frac{1}{h}\,\int_{\R^2} {\mathfrak S}\l(\frac{x+y}{2},\xi\r)\,e^{\frac{2i\pi(x-y)\xi}{h}}\,u(y)\,d\xi\,dy.
\end{equation}
In pure mathematics, even the more so when pseudodifferential analysis is applied to arithmetic, Planck's constant becomes a pure number: there is no question that the good such constant in ``pseudodifferential arithmetic'' is $2$, as especially put into evidence \cite[Chapter\,6]{birk18} in the pseudodifferential calculus of operators with automorphic symbols. We simplify as $\Psi$ the rule just introduced as $\Psi_2$ (denoted as ${\mathrm{Op}}_2$ in \cite[(2.1.1)]{birk18}).\\

We need to use for ``symbols'' ${\mathfrak S}$ elements of ${\mathcal S}'(\R^2)$. The convergence of the integral (\ref{21}), even in a weak sense, to wit after insertion under the integral of a factor $\overline{v}(x)\,dx$ with $v \in {\mathcal S}(\R)$, does not hold in general. However, with
\begin{equation}
(v\,|\,u)=\intR\overline{v}(x)\,u(x)\,dx
\end{equation}
for $v,u\in {\mathcal S}(\R)$ (note that this hermitian form is antilinear with respect to the factor on the left), one can use for the operator $\Psi({\mathfrak S})$ with symbol ${\mathfrak S}$ the weak definition
\begin{equation}
(v\,|\,\Psi({\mathfrak S})\,u)=\int_{\R^2} K(x,\,y)\,\overline{v}(x)\,u(y)\,dx\,dy,
\end{equation}
where the integral kernel $K(x,\,y)$ is the function
\begin{equation}\label{22}
K(x,y)=\hf\,\l({\mathcal F}_2^{-1}{\mathfrak S}\r)\l(\frac{x+y}{2},\,\frac{x-y}{2}\r),
\end{equation}
where ${\mathcal F}_2^{-1}$ denotes the inverse Fourier transformation with respect to the second variable.\\

Elementary cases of operators $\Psi({\mathfrak S})$ are the following. If ${\mathfrak S}(x,\xi)=f(x)$, the operator $\Psi({\mathfrak S})$ is the operator of multiplication by the function $f$. If ${\mathfrak S}(x,\xi)=\delta(x)\,g(\xi)$, one has $[\Psi({\mathfrak S})\,u](x)=({\mathcal F}^{-1}g)(x)\,u(-x)$.\\

If one defines the Wigner function $\Wig(v,u)$ of two functions in ${\mathcal S}(\R)$ as the function in  ${\mathcal S}(\R^2)$ such that
\begin{equation}\label{23}
\Wig(v,u)(x,\xi)=\intR \overline{v}(x+t)\,u(x-t)\,e^{2i\pi \xi t}dt,
\end{equation}
one has
\begin{equation}\label{24}
\l(v\,\bigr|\,\Psi({\mathfrak S})\,u\r)=\Lan {\mathfrak S},\,\Wig(v,u)\Ran.
\end{equation}
Yes, $\Wig(v,u)$ lies in ${\mathcal S}(\R^2)$ if $v$ and $u$ lie in ${\mathcal S}(\R)$:
the rapid decrease with respect to $\xi$ is obtained with the help of the integration by parts associated to the equation  $\xi\,e^{2i\pi \xi t}=\frac{1}{2i\pi}\,\frac{d}{dt}\,
e^{2i\pi \xi t}$. Immediately observe for future reference that if $v$ and $u$ are compactly supported, one can have $\Wig(v,u)(x,\xi)\neq 0$ only if $2x$ lies in the algebraic sum of the supports of $v$ and $u$.\\

If $v$ and $u$ have compact supports, the sequence $(v\,|\,\Psi({\mathfrak T}_N)\,u)$ is stationary as $N\nearrow \infty$ and one can write for $N$ ``sufficiently divisible''
\begin{equation}\label{412}
\l(v\,\big|\,\Psi({\mathfrak T}_{\infty})\,u\r)=\sum_{j,k \in \Z} a^{\flat}(N,\,j,\,k))\,
\Wig(v,\,u)(j,\,k).
\end{equation}\\

The following reduction, under some support conditions, of ${\mathfrak T}_{\infty}$ to ${\mathfrak T}_N$, is immediate, and fundamental for our purpose. Assume that $v$ and $u\in C^{\infty}(\R)$ are such that the algebraic sum of the supports of $v$ and $u$ is contained in $[0,2\beta]$. Then, given a squarefree integer $N=RQ$ (with $R,Q$ integers) divisible by all primes $<\beta Q$, one has
\begin{equation}\label{413}
\l(v\,\big|\,\Psi\l(Q^{2i\pi\E}{\mathfrak T}_{\infty}\r)u\r)=
\l(v\,\big|\,\Psi\l(Q^{2i\pi\E}{\mathfrak T}_N\r)u\r).
\end{equation}
Indeed, (\ref{A34}) and the fact that the transpose of $2i\pi\E$ is $-2i\pi\E$ yield
\begin{equation}\label{24third}
\l(v\,\big|\,\Psi\l(Q^{2i\pi\E}{\mathfrak T}_{\infty}\r)u\r)=Q^{-1}\sum_{j,k}a^{\flat}((j,k))\,\Wig(v,u)\l(\frac{j}{Q},\,\frac{k}{Q}\r).
\end{equation}
Next, from the observation that follows (\ref{24}), one has $0<\frac{j}{Q}<\beta$, or $0<j<\beta Q$, for all nonzero terms of this sum, which implies that all prime divisors of $j$ divide $N$, so that $a^{\flat}((j,k,N))=a^{\flat}((j,k))$. Note also that using here ${\mathfrak T}_N$ or ${\mathfrak T}_N^{\times}$ would not make any difference since $\Psi(\delta_0)\,u=\overset{\vee}{u}$ and the interiors of the supports of $v$ and $\overset{\vee}{u}$ do not intersect.\\

Given two functions in ${\mathcal S}(\R)$, the Wigner function of their Fourier transforms is
\begin{equation}\label{24half}
\Wig(\widehat{v},\,\widehat{u})(x,\xi)=\Wig(v,\,u)(\xi,\,-x).
\end{equation}
An incorrect proof of (\ref{24half}) (to make it correct, just use the definition of the Fourier transforms one at a time) goes as follows. The left-hand side of (\ref{24half}) is
\begin{multline}
\int_{\R^4} \overline{v}(r)\,u(s)\,\exp[2i\pi\,((x+t)r-(x-t)s)]\,e^{2i\pi t\xi}\,dr\,ds\,dt\\
=\int_{\R^2} \overline{v}(r)\,u(s)\,e^{2i\pi\,x(r-s)}\,\delta(r+s+\xi)\,dr\,ds,
\end{multline}
and setting $r+s=\xi,\,r-s=-\eta$, one obtains (\ref{24half}).\\

Another useful property of the calculus $\Psi$ is expressed by the following two equivalent identities, obtained with the help of elementary manipulations of the Fourier transformation or with that of (\ref{22}),
\begin{equation}\label{25}
\Psi\l(\Fymp{\mathfrak S}\r)\,w=
\Psi({\mathfrak S})\,\overset{\vee}{w},\qquad \Fymp\Wig(v,\,u)=\Wig(v,\,\overset{\vee}{u}),
\end{equation}
where $\overset{\vee}{w}(x)=w(-x)$ and the symplectic Fourier transformation in $\R^2$
is defined in ${\mathcal S}(\R^2)$ or ${\mathcal S}'(\R^2)$ by the equation
\begin{equation}\label{26}
\l(\Fymp\,{\mathfrak S}\r)(x,\xi)=\int_{\R^2} {\mathfrak S}(y,\eta)\,e^{2i\pi(x\eta-y\xi)}\,dy\,d\eta.
\end{equation}\\

Recall the definition $2i\pi\E=1+x\frac{\partial}{\partial x}+\xi\frac{\partial}{\partial \xi}$ of the
Euler operator and define, for $t>0$, the operator $t^{2i\pi\E}$ such that $\l(t^{2i\pi\E}{\mathfrak S}\r)(x,\xi)=t\,{\mathfrak S}(tx,t\xi)$. In the space $L^2(\R^2)$, the operator $\E$ is essentially self-adjoint and, according to Stone's theorem, this formula would be a theorem rather than a definition. But (this is a piece of good news), we shall never have to use in the plane spaces of functions or distributions besides ${\mathcal S}(\R^2)$ and ${\mathcal S}'(\R^2)$; the same holds in the case of one-dimensional functions.\\

The following lemma gives the operator $2i\pi\E$ a role in pseudodifferential analysis: simultaneous rescalings of $x,\xi$ by the same factor do not correspond under $\Psi$ to a pair of unitary transformations of the one-dimensional functions $v,\,u$, in contrast to the linear transformations associated to matrices in $SL(2,\R)$.\\

\begin{lemma}\label{lem21}
Given $v,u\in {\mathcal S}(\R)$, one has
\begin{equation}\label{29}
(2i\pi\E)\,{\mathrm{Wig}}(v,\,u)={\mathrm{Wig}}(v',\,xu)+{\mathrm{Wig}}(xv,\,u').
\end{equation}
If one introduces the (Heisenberg) operators $Q$ and $P$ on functions of $t\in \R$ defined as $(Qu)(t)=t\,u(t)$ and $(Pu)(t)=\frac{1}{2i\pi}\,u'(t)$, this is equivalent to the general identity
\begin{equation}\label{29b}
P\,\Psi({\mathfrak S})\,Q-Q\,\Psi({\mathfrak S})\,P=\Psi(\E\,{\mathfrak S}).
\end{equation}\\

\end{lemma}

\begin{proof}
One has $\xi\frac{\partial}{\partial \xi}\,e^{2i\pi t\xi}=t\,\frac{\partial}{\partial t}\,e^{2i\pi t\xi}$. The term $\xi\frac{\partial}{\partial \xi}\,{\mathrm{Wig}}(v,\,u)$ is obtained with the help of an integration by parts, noting that the transpose of the operator $t\,\frac{\partial}{\partial t}$ is $-1-t\,\frac{\partial}{\partial t}$. Overall, using (\ref{22}),
one obtains
\begin{equation}
\l[(2i\pi\E)\,{\mathrm{Wig}}(v,\,u)\r](x,\,\xi)=\intR A(x,\,t)\,e^{2i\pi t\xi}\,dt,
\end{equation}
with
\begin{align}
A(x,\,t)&=x\,[\overline{v'}(x+t)\,u(x-t)+\overline{v}(x+t)\,u'(x-t)]\nonumber\\
&+t\,[-\overline{v'}(x+t)\,u(x-t)+\overline{v}(x+t)\,u'(x-t)]\nonumber\\
&=(x-t)\,\overline{v'}(x+t)\,u(x-t)+(x+t)\,\overline{v}(x+t)\,u'(x-t).
\end{align}
The equation (\ref{29b}) is an immediate reformulation of (\ref{29}). While of limited interest in the present book, the identity (\ref{29b}) is decisive in our proof (arXiv:2001.10956) of the Ramanujan-Petersson conjecture for Maass forms.\\

Note that if one defines the ``mixed-adjoint'' operator ${\mathrm{mad}}(P\,\wedge\,Q)$
by the equation ${\mathrm{mad}}(P\,\wedge\,Q)\,A=P\,A\,Q-Q\,A\,P$ for every continuous linear operator from ${\mathcal S}(\R)$ to ${\mathcal S}'(\R)$, the exponentiated version $h^{-i\pi\,{\mathrm{mad}}(P\,\wedge\,Q)}$, with $h>0$, corresponds exactly, in the Weyl calculus, to a change of Planck's constant.\\
\end{proof}

The following theorem, especially (\ref{317}), is the main tool in all that follows.\\

\begin{theorem}\label{theo34}
For every tempered distribution ${\mathfrak S}$ such that ${\mathfrak S}\,\circ\,\sm{1}{0}{1}{1}={\mathfrak S}$, especially every automorphic distribution, the integral kernel of the operator $\Psi({\mathfrak S})$ is supported in the set of points $(x,y)$ such that $x^2-y^2\in 4\Z$. In particular, given $\nu\in \C, \,\nu\neq \pm 1$, one has for every pair $v,u$ in ${\mathcal S}(\R)$,
\begin{multline}\label{315half}
\l(v\,\bigr|\,\Psi\l({\mathfrak E}_{-\nu}\r)u\r)= \zeta(\nu)\,\l(|x|^{\nu-1}v\,\big|\,u\r) +\zeta(-\nu)\,\l(|x|^{-\nu-1}v\,\big|\,\overset{\vee}{u}\r)\\
+\sum_{r\neq 0}\sigma_{-\nu}(r)\,\intR \overline{v}\l(t+\frac{r}{t}\r)\,|t|^{\nu-1}\,u\l(t-\frac{r}{t}\r) dt.
\end{multline}
Under the support condition that $x>0$ and $0<x^2-y^2<8$ if $v(x)u(y)\neq 0$, one has
\begin{equation}\label{317}
\l(v\,\bigr|\,\Psi\l({\mathfrak E}_{-\nu}\r)u\r)= \int_0^{\infty} \overline{v}\l(t+t^{-1}\r) \,t^{\nu-1}\, u\l(t-t^{-1}\r) dt.
\end{equation}\\
\end{theorem}

\begin{proof}
The invariance of ${\mathfrak S}$ under the change $(x,\xi)\mapsto (x,\xi+x)$ implies that the distribution $\l({\mathcal F}_2^{-1}{\mathfrak S}\r)(x,z)$ is invariant under the multiplication by $e^{2i\pi xz}$. Applying (\ref{22}), we obtain the first assertion.\\

Let us use the expansion (\ref{37}), but only after we have substituted the pair $(\xi,-x)$ for $(x,\xi)$, which does not change ${\mathfrak E}_{-\nu}(x,\xi)$ for any $\nu\neq \pm 1$ in view of (\ref{A35}): hence,
\begin{equation}\label{318}
{\mathfrak E}_{-\nu}(x,\xi)=\zeta(\nu)\,|x|^{\nu-1}+\zeta(1+\nu)\,\delta(x)\,|\xi|^{\nu}
+\sum_{r\neq 0}\sigma_{-\nu}(r)\,|x|^{\nu-1}\exp\l(-2i\pi\frac{r\xi}{x}\r).
\end{equation}
The contributions to $\l(v\,\bigr|\,\Psi\l({\mathfrak E}_{-\nu}\r)u\r)$ of the first two terms of this expansion are obtained by what immediately followed (\ref{22}). The sum of terms of the series for $r\neq 0$, to be designated as ${\mathfrak E}_{-\nu}^{\mathrm{trunc}}(x,\xi)$, remains to be examined.\\

One has
\begin{equation}\label{322}
\l({\mathcal F}_2^{-1}{\mathfrak E}_{-\nu}^{\mathrm{trunc}}\r)(x,z)=
\sum_{r\neq 0}\sigma_{-\nu}(r)\,|x|^{\nu-1}\delta\l(z-\frac{r}{x}\r).
\end{equation}
Still using (\ref{22}), the integral kernel of the operator $\Psi({\mathfrak S}_r)$, with
\begin{equation}
{\mathfrak S}_r(x,\xi)\colon =|x|^{\nu-1}\exp\l(-2i\pi\frac{r\xi}{x}\r)={\mathfrak T}_r\l(x,\frac{\xi}{x}\r),
\end{equation}
is half of
\begin{align}\label{320}
2\,K_r(x,y)
&=\l({\mathcal F}_2^{-1}{\mathfrak S}_r\r)\l(\frac{x+y}{2},\frac{x-y}{2}\r)
=\l|\frac{x+y}{2}\r|\,\l({\mathcal F}_2^{-1}{\mathfrak S}_r\r)\l(\frac{x+y}{2},\,\frac{x^2-y^2}{4}\r)\nonumber\\
&=\l|\frac{x+y}{2}\r|^{\nu}\delta\l(\frac{x^2-y^2}{4}-r\r)=\l|\frac{x+y}{2}\r|^{\nu-1}
\delta\l(\frac{x-y}{2}-\frac{2r}{x+y}\r).
\end{align}
Making in the integral $\int_{\R^2}K_r(x,y)\,\overline{v}(x)\,u(y)\,dx\,dy$ the change of variable which amounts to taking $\frac{x+y}{2}$ and $x-y$ as new variables, one obtains (\ref{315half}).
Under the support assumptions made about $v,u$ in the second part, only the term such that $r=1$ subsists.\\
\end{proof}

\section{A criterion for the Riemann hypothesis}

Using (\ref{45}) and the homogeneity of ${\mathfrak E}_{-\nu}$, one has
\begin{equation}\label{51}
Q^{2i\pi\E}{\mathfrak T}_{\infty} =\frac{1}{2i\pi}\int_{\Re \nu=c} Q^{\nu}\,\,\frac{{\mathfrak E}_{-\nu}}{\zeta(\nu)}\,d\nu,\qquad c>1.
\end{equation}
From Lemma \ref{lem33} and (\ref{A315}), the product by $Q^{-1-\eps}$ of the distribution
$Q^{2i\pi\E}{\mathfrak T}_{\infty}$ remains for every $\eps>0$ in a bounded subset of ${\mathcal S}'(\R^2)$ as $Q\to \infty$. If the Riemann hypothesis holds, the same is true after
$Q^{-1-\eps}$ has been replaced by $Q^{-\hf-\eps}$.\\

We shall state and prove some suitably modified version of the converse, involving for some choice of the pair $v,\,u$ of functions in ${\mathcal S}(\R)$ the function
\begin{equation}\label{52}
F_0^{\flat}(v,\,u;\,s)=\sum_{Q\in \,{\Sqo}}\,Q^{-s}
\l(v\,\big|\,\Psi\l(Q^{2i\pi\E}{\mathfrak T}_{\frac{\infty}{2}}\r)u\r),
\end{equation}
where we denote as $\Sqo$ the set of squarefree odd integers. We are discarding the prime $2$ so as to avoid unnecessary complications in the sections to come on pseudodifferential arithmetic. Under the assumption that $\l(v\,\bigr|\,\Psi\l(Q^{2i\pi\E}{\mathfrak T}_{\infty}\r) u\r)=
{\mathrm{O}}\l(Q^{\hf+\eps}\r)$, the function $F_0^{\flat}(v,\,u;\,s)$ is analytic in the half-plane $\Re s>\frac{3}{2}$ and polynomially bounded in vertical strips in this domain, by which we mean, classically, that given $[a,b]\subset ]\frac{3}{2},\infty[$, there exists $M>0$ such that, for some $C>0$, $|F_0^{\flat}(\sigma+it)|\leq C\,(1+|t|)^M$ when $\sigma \in [a,b]$ and $t\in \R$.\\

First, we transform the series defining $F_0^{\flat}(v,\,u;\,s)$ into a line integral of convolution type.\\

\begin{lemma}\label{lem51}
Given $v,,u\in {\mathcal S}(\R)$, the function $F_0^{\flat}(v,\,u;\,s)$ introduced in {\em(\ref{52})\/} can be written for $\Re s$ large
\begin{multline}\label{53}
F_0^{\flat}(v,\,u;\,s)
=\frac{1}{2i\pi}\int_{\Re \nu=1}\l(1+2^{-s+\nu}\r)^{-1}
\frac{\zeta(s-\nu)}{\zeta(2(s-\nu))}\\
\times \,
\frac{\l(1-2^{-\nu}\r)^{-1}}{\zeta(\nu)}\,\Lan {\mathfrak E}_{-\nu},\,\Wig(v,\,u)\Ran\,
d\nu.
\end{multline}
\end{lemma}

\begin{proof}
For $\Re s>1$, one has the identity
\begin{align}\label{54}
\sum_{Q\in\,\Sqo}\,Q^{-s}&=\prod_{\begin{array}{c} p\,{\mathrm{prime}}\\ p\neq 2\end{array}}\l(1+p^{-s}\r)\nonumber\\
&=\l(1+2^{-s}\r)^{-1}\prod_p\frac{1-p^{-2s}}{1-p^{-s}}
=\l(1+2^{-s}\r)^{-1}\frac{\zeta(s)}{\zeta(2s)}.
\end{align}
We apply this with $s$ replaced by $s-\nu$, starting from (\ref{52}). One has $\l(v\,\big|\,\Psi\l(Q^{2i\pi\E}{\mathfrak E}_{-\nu}\r)u\r)=
Q^{\nu}\Lan {\mathfrak E}_{-\nu},\,\Wig(v,u)\Ran$ according to (\ref{24}). Using (\ref{51}) and (\ref{54}), one obtains (\ref{53}) if $\Re s>2$, noting that the denominator $\zeta(\nu)$ takes care of the pole of ${\mathfrak E}_{-\nu}$ at $\nu=1$, which makes it possible to replace the line $\Re\nu=c,\,c>1$ by the line $\Re\nu=1$.
Note that the integrability at infinity is taken care of by (\ref{A315}), together with Hadamard's theorem as recalled immediately after (\ref{412}).\\
\end{proof}

\begin{lemma}\label{lem52}
Let $\rho\in \C$ and consider a product $h(\nu)\,f(s-\nu)$, where the function $f=f(z)$, defined and meromorphic near the point $z=1$, has a simple pole at this point, and the function $h$, defined and meromorphic near $\rho$, has at that point a pole of order $\ell\geq 1$. Then, the function $s\mapsto {\mathrm{Res}}_{\nu=\rho}\l[h(\nu)\,f(s-\nu)\r]$ has at $s=1+\rho$ a pole of order $\ell$.\\
\end{lemma}

\begin{proof}
If $h(\nu)=\sum_{j=1}^{\ell}a_j(\nu-\rho)^{-j}+{\mathrm{O}}(1)$ as $\nu\to \rho$, one has for $s$ close to $1+\rho$ but distinct from this point
\begin{equation}\label{55}
{\mathrm{Res}}_{\nu=\rho}\l[h(\nu)\,f(s-\nu)\r]=\sum_{j=1}^{\ell}(-1)^{j-1}a_j
\frac{f^{(j-1)}(s-\rho)}{(j-1)\,!},
\end{equation}
and the function $s\mapsto f^{(j-1)}(s-\rho)$ has at $s=1+\rho$ a pole of order $j$.\\
\end{proof}

\noindent
{\em Remark\/} 4.1. An integral $\int_{\gamma}h(\nu)\,f(s-\nu)\,d\nu$ over a finite part of a line cannot have a pole at $s=s_0$ unless both $h$ and the function $\nu\mapsto h(s_0-\nu)$ have poles at the same point $\nu\in \gamma$. This follows from the residue theorem together with the case $\ell=0$ of Lemma \ref{lem52}.

\begin{theorem}\label{theo53}
Assume that, for every given pair $v,u$ of functions in ${\mathcal S}(\R)$, the function $F_0^{\flat}(v,\,u;\,s)$ defined in the integral version {\em(\ref{53})\/}, initially defined and analytic for $\Re s>2$, extends as an analytic function in the half-plane $\{s\colon \Re s>\frac{3}{2}\}$. Then, all zeros of zeta have real parts $\leq \hf$: in other words, the Riemann hypothesis does hold.\\
\end{theorem}

\begin{proof}
Assume that a zero $\rho$ of zeta with a real part $>\hf$ exists: one may assume that the real part of $\rho$ is the largest one among those of all zeros of zeta (if any other should exist !) with the same imaginary part and a real part $>\hf$. Choose $\beta$ such that $0<\beta<\Re(\rho-\hf)$. Assuming that $\Re s>2$, change the line $\Re \nu=1$ to a simple contour $\gamma$ on the left of the initial line, enclosing the point $\rho$ but no other point $\rho$ with $\zeta(\rho)=0$, coinciding with the line $\Re \nu=1$ for $|\Im \nu|$ large, and such that $\Re\nu>\Re\rho-\beta$ for $\nu\in\gamma$.\\

Let $\Omega$ be the relatively open part of the half-plane $\Re\nu\leq 1$ enclosed by $\gamma$ and the line $\Re\nu=1$. Let ${\mathcal D}$ be the domain consisting of the numbers $s$ such that $s-1\in \Omega$ or $\Re s>2$. When $s\in {\mathcal D}$, one has $\Re s>1+\Re\rho-\beta>\frac{3}{2}$.\\

Still assuming that $\Re s>2$, one obtains the equation
\begin{equation}\label{56}
F_0^{\flat}(v,\,u;\,s)
=\frac{1}{2i\pi}\int_{\gamma} f(s-\nu)\,h_0(\nu)\,d\nu
+{\mathrm{Res}}_{\nu=\rho}\l[\,h_0(\nu)f(s-\nu)\r],
\end{equation}
with
\begin{align}\label{57}
h_0(\nu)&=\l(1-2^{-\nu}\r)^{-1}\,\times\,\frac{\Lan {\mathfrak E}_{-\nu},\,\Wig(v,\,u)\Ran}{\zeta(\nu)},\nonumber\\
f(s-\nu)&=\l(1+2^{-s+\nu}\r)^{-1}\,\times\,\frac{\zeta(s-\nu)}{\zeta(2(s-\nu))}.
\end{align}\\

We show now that the integral term in (\ref{56}) is holomorphic in the domain ${\mathcal D}$. The first point is that the numerator $\zeta(s-\nu)$ of $f(s-\nu)$ will not contribute singularities. Indeed, one can have $s-\nu=1$ with $s\in {\mathcal D}$ and $\nu\in \gamma$ only if $s\in 1+\Omega$, since $\Re(s-\nu)>1$ if $\Re s>2$ and $\nu\in\gamma$. Then, the conditions $s-1\in \Omega$ and $s-1\in \gamma$ are incompatible because on one hand, the imaginary part of $s-1$ does not agree with that of any point
of the two infinite branches of $\gamma$, while the rest of $\gamma$ is a part of the boundary of $\Omega$. Finally, when $s-1\in \Omega$ and $\nu\in\gamma$, that $\zeta(2(s-\nu))\neq 0$ follows from the inequalities $\Re(s-\nu)\geq \Re \rho-\beta>\hf$, since $\Re s>1+(\Re\rho-\beta)$ and $\Re \nu\leq 1$.\\

Since $F_0^{\flat}(v,\,u;\,s)$ is analytic for $\Re s>\frac{3}{2}$, it follows that the residue present in (\ref{56}) extends as an analytic function of $s$ in ${\mathcal D}$. But an application of Lemma \ref{lem52}, together with the first equation (\ref{57}) and (\ref{317}), shows that this residue is singular at $s=1+\rho$ for some choice of the pair $v,u$. We have reached a contradiction.\\

\noindent
{\em Remark\/} 4.2. This remark should prevent a possible misunderstanding: we regard it as important since it will have to be referred to time and again. Though we are ultimately interested in a residue at $\nu=\rho$ and in the continuation of $F_0^{\flat}(v,\,u;\,s)$ near $s=1+\rho$, we have established (\ref{56}) under the assumption that $\Re s>2$, in which case $s-1$ does not lie in the domain covered by $\nu$ between the line $\Re \nu=1$ and the line $\gamma$. We must therefore not add to the right-hand side of (\ref{56}) the residue of the integrand at $\nu=s-1$ (the two residues would have killed each other). Separating the poles of the two factors has been essential. The conclusion resulted from analytic continuation and the assumption that $F_0^{\flat}(v,\,u;\,s)$ extends as an analytic function for $\Re s>\frac{3}{2}$.\\
\end{proof}

If, for every pair $v,u$ in ${\mathcal S}(\R)$ and every $\eps>0$, one has the estimate
$\l(v\,\big|\,\Psi(Q^{2i\pi\E}{\mathfrak T}_{\infty})\,u\r)={\mathrm O}(Q^{\hf+\eps})$,
the Riemann hypothesis follows. More precisely,\\

\begin{corollary}\label{cor54}
Let $\rho\in \C$ with $\Re\rho>0$ be given. Let $v,\,u\in {\mathcal S}(\R)$ be such that
$\l(v\,\big|\,\Psi({\mathfrak E}_{-\rho})\,u\r)=\Lan {\mathfrak E}_{-\rho},\,\Wig(v,u)\Ran \neq 0$.
If the function $F_0^{\flat}(v,\,u;\,s)$, defined for $\Re s>2$ by {\em(\ref{52})\/}, can be continued analytically along a path connecting the point $1+\rho$ to a point with a real part $>2$, the point $\rho$ cannot be a zero of zeta. For any $\rho$, a pair $v,u$ with $v$ supported in $[2,\sqrt{8}]$ and $u$ supported in $[0,1]$, such that $\l(v\,\big|\,\Psi({\mathfrak E}_{-\rho})\,u\r)\neq 0$, can be found.
\end{corollary}

\begin{proof}
It follows the proof of Theorem \ref{theo53}. It suffices there to take $\Omega$ containing the path in the assumption, and to observe that $h_0(\nu)$, as defined in (\ref{57}), would under the nonvanishing condition $\Lan {\mathfrak E}_{-\rho},\,\Wig(v,u)\Ran \neq 0$ have a pole at $\rho$ if one had $\zeta(\rho)=0$.\\

For the second assertion of the corollary, we use (\ref{317}), concentrating the supports of $v$ and $u$ near a pair $(\sqrt{y_0^2+4},y_0)$ with $0<y_0<1$.\\
\end{proof}

The criterion was obtained in \cite[section\,3.4]{birk18}. At the time, we believed (like a majority of mathematicians) that the Riemann hypothesis did hold.
We also (wrongly) believed that the sole remaining difficulty before a proof of R.H. was obtained was to find an explicit computation of the main hermitian form present in (\ref{52}). Three years ago, we succeeded in making this computation: we then realized up to which point our optimism had been misplaced.\\

Though not quite as important as we believed then, the criterion has been the occasion to introduce some of the concepts, facts and methods essential in any case. To list a few: 1. giving ${\mathfrak T}_{\infty}$, through the Weyl calculus, a hermitian structure; 2. defining a ``main hermitian form''
$(v\,|\,\Psi(Q^{2i\pi \E}{\mathfrak T}_{\infty})\,u)$ and giving a criterion for the validity of R.H. in terms of bounds for it; 3. reducing under some conditions the analysis of ${\mathfrak T}_{\infty}$ to that of ${\mathfrak T}_N$.\\

We shall not truly use Theorem \ref{theo53} as such: to start with, we shall not prove R.H., on the contrary. In this direction, the function $F_0^{\,\flat}(v,\,u;\,s)$, certainly a model for all which will follow, will have to give way, for the application to Lindel\"of's hypothesis, to the function $F_0^{\flat}(v,\,u;\,s)$ obtained from $F_0^{\,\flat}(v,\,u;\,s)$ when replacing the coefficient $a^{\flat}(r)$ by $a(r)$.\\

The reduction of the analysis of ${\mathfrak T}_{\infty}$ to that of ${\mathfrak T}_N$
gives significance to the congruence-theoretic developments in the three sections to come: you may, or not, agree with the name of ``pseudodifferential arithmetic'' we gave them.\\

\section{Pseudodifferential arithmetic}

In this section, we make no support assumptions on $v,u$, just taking them in ${\mathcal S}(\R)$.
We consider operators of the kind $\Psi(Q^{2i\pi\E}{\mathfrak S})$ with
\begin{equation}\label{61}
{\mathfrak S}(x,\xi)=\sum_{j,k\in \Z} b(j,k)\,\delta(x-j)\,\delta(\xi-k),
\end{equation}
under the following assumptions: that $N=RQ$ with $(R,Q)_=1$, and that $b$ satisfies the periodicity conditions
\begin{equation}\label{62}
b(j,k)=b(j+N,k)=b(j,k+N).
\end{equation}
Special cases consist of course of the symbols ${\mathfrak S}={\mathfrak T}_N$ or ${\mathfrak S}_N$.
The aim is to transform the hermitian form associated to the operator $\Psi\l(Q^{2i\pi\E}{\mathfrak T}_N\r)$ to an arithmetic version.\\

The following theorem reproduces \cite[Prop.\,4.1.2,\,\,Prop.\,4.1.3]{birk18}, the parameter denoted as $\omega$ there being set to the value $2$.\\

\begin{theorem}\label(A.Theorem \,6.1)
With $N=RQ$ and $b(j,k)$ satisfying the condition {\em(\ref{62})\/}, define the function
\begin{equation}\label{63}
f_N(j,\,s)=\frac{1}{N}\sum_{k\mm N} b(j,k)\,\exp\l(\frac{2i\pi ks}{N}\r),\,\qquad j,s\in \Z/N\Z.
\end{equation}
Set, noting that the condition $m-n\equiv 0\mm 2Q$ implies that $m+n$ too is even,
\begin{equation}\label{64}
c_{R,Q}\l({\mathfrak S};\,m\,,n\r)
={\mathrm{char}}(m+n\equiv 0\mm R,\,m-n\equiv 0\mm 2Q)\,f_N\l(\frac{m+n}{2R},\,\frac{m-n}{2Q}\r).
\end{equation}
On the other hand, set, for $u\in {\mathcal S}(\R)$,
\begin{equation}\label{65}
(\theta_Nu)(n)=\sum_{\ell \in \Z} u\l(\frac{n}{N}+2\ell N\r),\qquad n\mm 2N^2.
\end{equation}
Then, if $v,u\in {\mathcal S}(\R)$, one has with ${\mathfrak S}$ as defined in (\ref{61})
\begin{equation}\label{66}
\l(v\,\bigr|\,\Psi\l(Q^{2i\pi\E}{\mathfrak S}\r)u\r)=\sum_{m,n\in\Z/(2N^2)\Z}
c_{R,Q}\l({\mathfrak S};\,m\,,n\r)\,\overline{\theta_Nv(m)}\,(\theta_Nu)(n).
\end{equation}\\
\end{theorem}

\begin{proof}
There is no restriction here on the supports of $v,u$, and one can replace these two functions by $v[Q]\,,u[Q]\,$ defined as $v[Q]\,(x)=v(Qx)$ and $u[Q]\,(x)=u(Qx)$\,. One has
\begin{equation}\label{67}
(\theta_N\,u[Q]\,)(n)=(\kappa u)(n)\colon =\sum_{\ell \in \Z} u\l(\frac{n}{R}+2QN\ell\r),\qquad n\mm 2N^2.
\end{equation}
It just requires the definition (\ref{21}) of $\Psi$ and changes of variables amounting to rescaling $x,y,\xi$ by the factor $Q^{-1}$ to obtain
$\l(v[Q]\,\,\big|\,\Psi\l(Q^{2i\pi\E}{\mathfrak S}\r) u[Q]\,\r)=(v\,|\,{\mathcal B}u)$, with
\begin{equation}\label{68}
\l({\mathcal B}\,u\r)(x)=\frac{1}{2\,Q^2}\int_{\R^2}{\mathfrak S}\l(\frac{x+y}{2},\,\xi\r) u(y)\,\exp\l(\frac{i\pi}{Q^2}(x-y)\xi\r) dy\,d\xi.
\end{equation}

The identity (\ref{66}) to be proved is equivalent, with $\kappa$ as defined in (\ref{67}), to
\begin{equation}\label{69}
(v\,|\,{\mathcal B}u)=\sum_{m,n\in\Z/(2N^2)\Z}c_{R,Q}\l({\mathfrak S};\,m\,,n\r)\,
\overline{\kappa v(m)}\,(\kappa u)(n).
\end{equation}

From (\ref{64}), one has
\begin{align}\label{610}
(v\,|\,{\mathcal B}u)&=\frac{1}{2Q^2}\intR\overline{v}(x)\,dx\int_{\R^2}{\mathfrak S}\l(\frac{x+y}{2},\,\xi\r) u(y)\,\exp\l(\frac{i\pi}{Q^2}(x-y)\xi\r) dy\,d\xi\nonumber\\
&=\frac{1}{Q^2}\intR\overline{v}(x)\,dx\int_{\R^2}{\mathfrak S}(y,\xi)\,u(2y-x)\,
\exp\l(\frac{2i\pi}{Q^2}(x-y)\xi\r) dy\,d\xi\nonumber\\
&=\frac{1}{Q^2}\intR\overline{v}(x)\,dx\sum_{j,k\in\Z} b(j,k)\,u(2j-x)\,\exp\l(\frac{2i\pi}{Q^2}(x-j)k\r).
\end{align}
Since $b(j,k)=b(j,k+N)$, one replaces $k$ by $k+N\ell$, the new $k$ lying in the interval $[0,N-1]$ of integers. One has (Poisson's formula to the effect that $\sum_{\ell}e^{2i\pi\alpha x}=|\alpha|^{-1}\sum_{\ell} \delta(x-\frac{\ell}{\alpha})$)
\begin{equation}\label{611}
\sum_{\ell \in \Z}\exp\l(\frac{2i\pi}{Q^2}(x-j)\ell N\r)=
\sum_{\ell \in \Z}\exp\l(\frac{2i\pi}{Q}(x-j)\ell R\r)
=\frac{Q}{R}\sum_{\ell \in \Z} \delta\l(x-j-\frac{\ell Q}{R}\r),
\end{equation}
and, from (\ref{610}),
\begin{multline}\label{612}
({\mathcal B}u)(x)\\
=\frac{1}{N}\sum_{\begin{array}{c}j\in \Z\\ 0\leq k<N\end{array}} b(j,k) \sum_{\ell\in \Z} u\l(j-\frac{\ell Q}{R}\r)\,\exp\l(\frac{2i\pi(x-j)k}{Q^2}\r) \delta\l(x-j-\frac{\ell Q}{R}\r)\\
=\sum_{m\in \Z} t_m\,\delta\l(x-\frac{m}{R}\r),
\end{multline}
with $m=Rj+\ell Q$ and $t_m$ to be made explicit: we shall drop the summation with respect to $\ell$ for the benefit of a summation with respect to $m$, an integer constrained by the sole condition $m\equiv Rj\mm Q$. Since, when $x=j+\frac{\ell Q}{R}=\frac{m}{R}$, one has $\frac{x-j}{Q^2}=\frac{\ell}{N}=\frac{m-Rj}{NQ}$ and $j-\frac{\ell Q}{R}=2j-x=2j-\frac{m}{R}$, one has
\begin{multline}
t_m\\
=\frac{1}{N}\sum_{\begin{array}{c}j\in \Z\\ 0\leq k<N\end{array}} b(j,k) \,{\mathrm{char}}(m\equiv Rj\mm Q) u\l(2j-\frac{m}{R}\r) \exp\l(\frac{2i\pi k(m-Rj)}{NQ}\r).\\
\end{multline}
After one has changed $j$ to $j+\ell_1QN$ with $\ell_1\in \Z$, the new $j$ hying in the interval $[0,QN[$, this transforms to
\begin{multline}
t_m=\frac{1}{N}\sum_{\begin{array}{c}0\leq j <QN\\ 0\leq k<N\end{array}} b(j,k) \,{\mathrm{char}}(m\equiv Rj\mm Q)\\
\sum_{\ell_1\in \Z} u\l(2(j+\ell_1QN)-\frac{m}{R}\r)\,\exp\l(\frac{2i\pi k(m-Rj)}{QN}\r).
\end{multline}
Recalling the definition (\ref{67}) of $\kappa u$, one obtains
\begin{multline}\label{614}
t_m=\frac{1}{N}\sum_{\begin{array}{c} 0\leq j<QN\\ 0\leq k<N\end{array}}
b(j,k)\,{\mathrm{char}}(m\equiv Rj\mm Q)\\
(\kappa u)(2Rj-m)\,\exp\l(\frac{2i\pi k(m-Rj)}{QN}\r).
\end{multline}
Using (\ref{612}), we obtain
\begin{multline}\label{615}
(v\,|\,{\mathcal B}u)=\frac{1}{N}\sum_{0\leq j<QN}\sum_{0\leq k<N} b(j,k)\,\sum_{\begin{array}{c} m_1\in \Z \\ m_1\equiv Rj\mm Q\end{array}} \\ \overline{v}\l(\frac{m_1}{R}\r)\,(\kappa u)(2Rj-m_1)\,\exp\l(\frac{2i\pi k(m_1-Rj)}{QN}\r).
\end{multline}
The change of $m$  to $m_1$ is just a change of notation.\\

Fixing $k$, we trade the set of pairs $m_1,j$ with $(m_1\in \Z,\,0\leq j< RQ^2,\,m_1\equiv Rj\mm Q)$ for the set of pairs $m,n\in \l(\Z/(2N^2)\Z\r)\times \l(\Z/(2N^2)\Z\r)$, where $m$ is the class mod $2N^2$ of $m_1$ and $n$ is the class mod $2N^2$ of $2Rj-m_1$. Of necessity, $m+n\equiv 0\mm 2R$ and $m-n\equiv 2(m-Rj)\equiv 0\mm 2Q$. Conversely, given a pair of classes $m,n$ mod $2N^2$ satisfying these conditions,
the equation $2Rj-m=n$ uniquely determines $j$ mod $\frac{2N^2}{2R}=RQ^2$, as it should. The sum $\sum_{m_1\equiv m\mm 2N^2} v\l(\frac{m_1}{R}\r)$ is just $(\kappa v)(m)$, and we have obtained the identity
\begin{equation}
(v\,|\,{\mathcal B}u)=\sum_{m,n\mm 2N^2} c_{R,Q}\l({\mathfrak S};\,m\,,n\r)\,\overline{(\kappa v)(m)}\,(\kappa u)(n),
\end{equation}
provided we define
\begin{multline}
c_{R,Q}\l({\mathfrak S};\,m\,,n\r)=\frac{1}{N}\,{\mathrm{char}}(m+n\equiv 0\mm R,\,\,m-n\equiv 0\mm 2Q)\\
\sum_{k\mm N} b\l(\frac{m+n}{2R},\,k\r) \exp\l(\frac{2i\pi k}{N}\,\frac{m-n}{2Q}\r),
\end{multline}
which is just the way indicated in (\ref{63}),\,(\ref{64}).\\
\end{proof}

\section{Computation of the arithmetic side of the main identity}\label{seccompu}

The calculations in this section and the one that follows are just the same if addressing to the distribution ${\mathfrak T}_N$ or the distribution ${\mathfrak S}_N$: the only difference ((\ref{A32}),\,(\ref{A33})) is that, changing ${\mathfrak T}_N$ to ${\mathfrak S}_N$, we must replace everywhere the M\"obius indicator $\mu(r)$ by its absolute value $|\mu(r)|$, the characteristic function of squarefree integers. Restricting oneself to odd values of $N$, a harmless restriction, dispenses us with a separation of cases as was necessary in \cite[section\,4.1]{birk18}

\begin{lemma}\label{lem71}
With the notation of Theorem {\em\ref(A.Theorem \,6.1)\/}, one has if $N=RQ$ is squarefree odd
\begin{multline}\label{71}
c_{R,Q}({\mathfrak T}_N;\,m,\,n)={\mathrm{char}}(m+n\equiv 0\mm 2R)\,{\mathrm{char}}(m-n\equiv 0\mm 2Q)\\
\sum_{\begin{array}{c} R_1R_2=R \\ Q_1Q_2=Q \end{array}}
\mu(R_1Q_1)\,{\mathrm{char}}\l(\frac{m+n}{R}\equiv 0 \mm R_1Q_1\r)\,
{\mathrm{char}}\l(\frac{m-n}{2Q}\equiv 0 \mm R_2Q_2\r).
\end{multline}\\
\end{lemma}

\begin{proof}
We compute the function $f_N(j,\,s)$ defined in (\ref{63}) in association to $b(j,\,k)=a^{\flat}((j,k))$. If $N=N_1N_2$ and $cN_1+dN_2=1$, so that $\frac{k}{N}=\frac{dk}{N_1}+\frac{ck}{N_2}$, one identifies $k\in \Z/N\Z$ with the pair $(k_1,k_2)\in
\Z/N_1\Z\times \Z/N_2\Z$ such that $k_1\equiv dk\mm N_1,\,k_2\equiv ck\mm N_2$. On one hand,
\begin{equation}
\exp\l(\frac{2i\pi ks}{N}\r)=\exp\l(\frac{2i\pi k_1s}{N_1}\r)\,\times\,\exp\l(\frac{2i\pi k_2s}{N_2}\r).
\end{equation}
On the other hand, as $(d,N_1)=(c,N_2)=1$,
\begin{equation}\label{73}
a^{\flat}((j,k,N))=a^{\flat}((j,k,N_1))\,a^{\flat}((j,k,N_2))=a^{\flat}((j,k_1,N_1))\,a^{\flat}((j,k_2,N_2)),
\end{equation}
and the same goes if replacing everywhere $a^{\flat}(...)$ by $a(...)$.
It follows from (\ref{63}) that in both cases $f_N(j,\,s)=f_{N_1}(j,\,s)\,f_{N_2}(j,\,s)$, and one has the Eulerian formula $f_N=\otimes_{p|N} f_p$, with, in the case of ${\mathfrak T}_N$ only,
\begin{align}\label{74}
f_p(j,\,s)&=\frac{1}{p}\sum_{k \mm p} \l(1-p\,\,{\mathrm{char}}(j\equiv k\equiv 0\mm p)\r)\,\exp\l(\frac{2i\pi ks}{p}\r)\nonumber\\
&=\frac{1}{p}\sum_{k\mm p}\exp\l(\frac{2i\pi ks}{p}\r)-{\mathrm{char}}(j\equiv 0\mm p)\nonumber\\
&={\mathrm{char}}(s\equiv 0\mm p)-{\mathrm{char}}(j\equiv 0\mm p).
\end{align}
If dealing with ${\mathfrak S}_N$ in place of ${\mathfrak T}_N$, one must replace the minus sign in the third line of (\ref{74}) by the plus sign: the effect in the equation that follows is to replace $\mu(N_1)$ by its absolute value. Expanding the product,
\begin{equation}
f_N(j,\,s)=\sum_{N_1N_2=N} \mu(N_1)\,
{\mathrm{char}}(s\equiv 0\mm N_2)\,{\mathrm{char}}(j\equiv 0\mm N_1).
\end{equation}
The equation (\ref{71}) follows from (\ref{64}).\\
\end{proof}

\begin{theorem}\label{theo72}
Let $N=RQ$ be squarefree odd. Let $v,u\in C^{\infty}(\R)$, compactly supported, satisfying the conditions that $x>0$ and $0<x^2-y^2<8$ when $v(x)u(y)\neq 0$. Then, if $N$ is large enough,
\begin{align}\label{76}
\l(v\,\big|\,\Psi(Q^{2i\pi\E}{\mathfrak T}_N)\,u\r)&=\sum_{Q_1Q_2=Q} \mu(Q_1)\nonumber\\
&\sum_{R_1|R} \mu(R_1)\,\overline{v}\l(\frac{R_1}{Q_2}+\frac{Q_2}{R_1}\r)\,
u\l(\frac{R_1}{Q_2}-\frac{Q_2}{R_1}\r).
\end{align}\\
\end{theorem}

\begin{proof}
Characterize $n$ mod $2N^2$ by $n\in \Z$ such that $-N^2\leq n<N^2$. The equation $\l(\theta_Nu\r)(n)=\sum_{\ell\in \Z}u\l(\frac{n}{N}+2\ell N\r)$ imposes $\ell=0$, so that $\l(\theta_Nu\r)(n)=u\l(\frac{n}{N}\r)$. Similarly, $\l(\theta_Nv\r)(m)=v\l(\frac{m}{N}\r)$ if $-N^2\leq m<N^2$.\\

The identity (\ref{66}) and Lemma \ref{lem71} yield
\begin{multline}
\l(v\,\big|\,\Psi(Q^{2i\pi\E}{\mathfrak T}_N)\,u\r)=
\sum_{\begin{array}{c} R_1R_2=R \\ Q_1Q_2=Q \end{array}} \mu(R_1Q_1)\,
{\mathrm{char}}(m+n\equiv 0\mm 2RR_1Q_1)\\
{\mathrm{char}}(m-n\equiv 0\mm 2QR_2Q_2)\, \overline{v}\l(\frac{m}{N}\r)\,u\l(\frac{n}{N}\r).
\end{multline}

Set $m+n=(2RR_1Q_1)\,a,\,m-n=(2QR_2Q_2)\,b$, with $a,\,b\in \Z$. Then,
\begin{equation}
\frac{m^2}{N^2}-\frac{n^2}{N^2}=\frac{4QR(R_1R_2)(Q_1Q_2)\,ab}{N^2}=4ab.
\end{equation}
For all nonzero terms of the last equation, one has $0<\frac{m^2}{N^2}-\frac{n^2}{N^2}<8$. On the other hand, the symbol $\l(Q^{2i\pi\E}{\mathfrak T}_N\r)(x,\,\xi)=Q\,{\mathfrak T}_N(Qx,\,Q\xi)$, just as the symbol ${\mathfrak T}_N$, is invariant under the change of $x,\xi$ to $x,x+\xi$. It follows from Theorem \ref{theo34} that its integral kernel $K(x,\,y)$ is supported in the set of points $(x,\,y)$ such that $x^2-y^2\in 4\Z$. With $x=\frac{m}{N}$ and $y=\frac{n}{N}$, the only possibility is to take $x^2-y^2=4$, hence $ab=1$, finally $a=b=1$  since $a>0$.\\

The equations $m+n=(2RR_1Q_1),\,m-n=(2QR_2Q_2)$ give $\frac{m}{N}=\frac{R_1}{Q_2}+\frac{Q_2}{R_1}$ and
$\frac{n}{N}=\frac{R_1}{Q_2}-\frac{Q_2}{R_1}$.\\
\end{proof}

\section{Arithmetic significance of last theorem}

In view of the central role of Theorem \ref{theo72} in the proof to follow of the Riemann hypothesis, we give now an independent proof of the major part of it. Reading this section is thus in principle unnecessary: but we find the role of reflections in the equation (\ref{82}) below illuminating.\\

\begin{theorem}
Let $N=RQ$ be a squarefree odd integer. Introduce the reflection $n\mapsto \overset{\vee}{n}$ of $\Z/(2N^2)\Z$ such that $\overset{\vee}{n}\equiv n\mm R^2$ and $\overset{\vee}{n}\equiv -n \mm 2Q^2$. Then, with the notation in Theorem {\em\ref(A.Theorem \,6.1)\/}, one has
\begin{equation}\label{81}
c_{R,Q}\l({\mathfrak T}_N;\,m,\,n\r)=\mu(Q)\,c_{N,1}\l({\mathfrak T}_N;\,m,\,\overset{\vee}{n}\r).
\end{equation}
If two functions $u$ and $\widetilde{u}$ in ${\mathcal S}(\R)$ are such that
$\l(\theta_N\,\widetilde{u}\r)(n)=(\theta_Nu)(\overset{\vee}{n})$ for every $n\in \Z$, one has
\begin{equation}\label{82}
\l(v\,\big|\,\Psi\l(Q^{2i\pi\E}{\mathfrak T}_N\r) u\r)=\mu(Q)\,\l(v\,\big|\,\Psi({\mathfrak T}_N)\,\widetilde{u}\r),
\end{equation}\\
\end{theorem}

\begin{proof}
With the notation of Theorem \ref(A.Theorem \,6.1), and making use of (\ref{74}), one has the Eulerian decomposition $f_N=\otimes f_p$, in which, for every $p$, $f_p(j,\,s)=-f_p(s,\,j)$: it follows that
$f_N(j,\,s)=\mu(N)\,f_N(s,\,j)$. If dealing with ${\mathfrak S}_N$ in place of ${\mathfrak T}_N$, just write $f_p(j,\,s)=f_p(s,\,j)$ instead. The net result will be to replace everywhere the $\mu$-factors by their absolute values.\\

To prove (\ref{81}), one may assume that $R=1,\,N=Q$ since the $R$-factor is left unaffected by the map $n\mapsto \overset{\vee}{n}$. Then,
\begin{align}
c_{1,Q}\l({\mathfrak S};\,m,\,n\r)&={\mathrm{char}}(m-n\equiv 0\mm 2Q)\,f\l(\frac{m+n}{2},\,\frac{m-n}{2Q}\r),\nonumber\\
c_{Q,1}\l({\mathfrak S};\,m,\,\overset{\vee}{n}\r)&={\mathrm{char}}(m+\overset{\vee}{n}\equiv 0\mm 2Q)\,f\l(\frac{m+\overset{\vee}{n}}{2Q},\,\frac{m-\overset{\vee}{n}}{2}\r)\nonumber\\
&={\mathrm{char}}(m-n\equiv 0\mm 2Q)\,f\l(\frac{m-n}{2Q},\,\frac{m+n}{2}\r).
\end{align}
That the first and third line are the same, up to the factor $\mu(Q)$, follows from the set of equations (\ref{74}) $f_p(j,\,s)=-f_p(s,\,j)$.\\

The equation (\ref{82}) follows from (\ref{81}) in view of (\ref{66}).\\
\end{proof}

It is not immediately obvious that, given $u\in {\mathcal S}(\R)$, there exists $\widetilde{u}\in {\mathcal S}(\R)$ such that $\l(\theta_N\,\widetilde{u}\r)(n)=(\theta_Nu)(\overset{\vee}{n})$ for every $n\in \Z$. Note that such a function $\widetilde{u}$ could not be unique, since any translate by a multiple of $2N^2$ would do just as well.\\

\begin{lemma}
Given $u\in {\mathcal S}{\R}$, define
\begin{equation}\label{84}
\widetilde{u}(y)=\frac{1}{Q^2}\sum_{0\leq \sigma,\tau<Q^2} u\l(y+\frac{2R\tau}{Q}\r)\,
\exp\l(2i\pi\,\frac{\sigma(Nx+R^2\tau)}{Q^2}\r).
\end{equation}
Then, $\l(\theta_N\,\widetilde{u}\r)(n)=(\theta_Nu)(\overset{\vee}{n})$.\\
\end{lemma}

\begin{proof}
With such a definition, one has
\begin{multline}
(\theta_N\widetilde{u})(n)\\
=\frac{1}{Q^2}\sum_{\ell \in \Z} \sum_{0\leq \sigma,\tau<Q^2}
u\l(\frac{n}{N}+\frac{2R\tau}{Q}+\ell N\r) \,\exp\l(\frac{2i\pi\sigma}{Q^2}\,(n+\ell N^2+R^2\tau)\r).
\end{multline}
Summing with respect to $\sigma$, this is the same as
\begin{equation}
\sum_{\ell \in \Z} u\l(\frac{n+2R^2\tau+\ell N^2}{N}\r),
\end{equation}
where the integer $\tau \in [0,Q^2[$ is characterized by the condition $n+\ell N^2+2R^2\tau\equiv 0 \mm Q^2$, or $n+2R^2\tau\equiv 0\mm Q^2$. Finally, as $\ell \in \Z$, the number $n+2R^2\tau+\ell N^2$ runs through the set of integers $n_2$ such that $n_2\equiv -n\mm Q^2$ and $n_2\equiv n\mm 2R^2$, in other words the set of numbers $n_2$ such that $n_2\equiv \overset{\vee}{n} \mm 2N^2$.\\
\end{proof}

Proposition \ref{prop83} and Lemma \ref{lem84} to follow provide a proof of Theorem \ref(A.Theorem \,8.1) independent of the results of Sections 6 and 7.\\

\begin{proposition}\label{prop83}
Fixing the decomposition $N=RQ$ of a squarefree odd number $N$, and decomposing for $r=1,2,\dots$ the coefficient $a^{\flat}((r,N))$ as $a_R(r)a_Q(r)$, with $a_R(r)=a^{\flat}((r,R))$ and $a_Q(r)=a^{\flat}((r,Q))$, introduce the coefficients
\begin{equation}\label{87}
\overset{\vee}{a}(j,k)=a_R((j,k))\,\times\,\frac{1}{Q}\sum_{s,m\mm Q} a_Q((s,m))\,
\exp\l(\frac{2i\pi}{Q}\l(mj-ks\r)\r).
\end{equation}
Also, recalling that
\begin{equation}
{\mathfrak T}_N(x,\,\xi)=
\sum_{j,k\in\Z} a^{\flat}((j,k,N))\,\delta(x-j)\, \delta(\xi-k),
\end{equation}
set
\begin{equation}\label{89}
\overset{\vee}{\mathfrak T}_N(x,\,\xi)=
\sum_{j,k\in\Z} \overset{\vee}{a}(j,k)\,\delta(x-j)\, \delta(\xi-k).
\end{equation}\\

Then, given $,u\in {\mathcal S}(\R)$ and recalling the definition {\em(\ref{84})\/} of $\widetilde{u}$, one has
\begin{equation}\label{810}
\l(v\,\big|\,\Psi\l(Q^{2i\pi\E}\,\overset{\vee}{\mathfrak T}_N\r)\,u\r)
=\l(v\,\big|\,\Psi\l({\mathfrak T}_N\r)\,\widetilde{u}\r).
\end{equation}
\end{proposition}

\begin{proof}
Setting $\lambda=\frac{R}{Q}$, one has
\begin{equation}
\widetilde{u}(x)=\frac{1}{Q^2}\sum_{0\leq \sigma,\tau<Q^2} u(x+2\lambda\tau)\,\exp\l(2i\pi\sigma(\lambda x+\lambda^2\tau)\r).
\end{equation}
One has, changing $t$ to $t+\lambda \tau$ in the first integral to follow,
\begin{multline}
\Wig(v,\,\widetilde{u})(j,k)=\intR \overline{v}(j+t)\,\widetilde{u})(j-t)\,e^{2i\pi kt}dt\\
=\frac{1}{Q^2}\intR \overline{v}(j+\lambda\tau+t)\sum_{0\leq \sigma,\tau<Q^2} u(j+\lambda\tau-t)\,
\exp\l[2i\pi \sigma(\lambda(j-t))\r] \exp\l(2i\pi k(\lambda\tau+t)\r) dt.
\end{multline}

The right-hand side of (\ref{810}) can thus be written as
\begin{multline}\label{813}
\frac{1}{Q^2}\sum_{j,k\in\Z} b_R(j,k)\,b_Q(j,k)\sum_{0\leq \sigma,\tau<Q^2}
\exp\l(2i\pi\lambda(k\tau+j\sigma)\r)\\
\intR \overline{v}(j+\lambda\tau+t)\,u(j+\lambda\tau-t)\,\exp\l[2i\pi t(k-\sigma\lambda)\r] dt.
\end{multline}\\\

Using that $Q\lambda\in \Z$ and $N-Q^2\lambda=0$, one observes that
this expression is invariant under the change $\sigma\mapsto \sigma+Q^2$, as seen if one accompanies this change by the change $k\mapsto k+M$; similarly, it is invariant under the change $\tau\mapsto \tau+Q^2$ accompanied by the change $j\mapsto j-N$. When $j,k$ have been given, or only given mod $N$, one can thus regard in this expression --- but not in the definition of the map $u\mapsto \widetilde{u}$ --- the pairs $\sigma,\tau$ as pairs of classes mod $Q^2$.\\

Defining $m=Qj+R\tau,\,n=Qk-R\sigma$, we shall reorganize this sum as a sum over $m,n\in \Z$. Note that $j+\lambda \tau=\frac{m}{Q}$ and that $k-\sigma \lambda=\frac{n}{Q}$, so that the integral can be rewritten as
\begin{equation}
\intR \overline{v}\l(\frac{m}{Q}+t\r)\,u\l(\frac{m}{Q}-t\r)\,\exp\l(\frac{2i\pi\,nt}{Q}\r) dt.
\end{equation}
Given $m,n$, let $\sigma_0,\tau_0$ be the solutions of the congruences $R\sigma_0\equiv -n\mm Q^2,\,R\tau_0\equiv m\mm Q^2$ lying in the interval $[0,Q^2[$ of integers. One has of necessity $\sigma\equiv \sigma_0\mm Q,\,\tau\equiv \tau_0\mm Q$, and one must thus have $\sigma=\sigma_0+Q\sigma_*,\,\tau=\tau_0+Q\tau_*$ for some pair $\sigma_*,\tau_*$ of integers.
When $m,n$ have been fixed, knowing $\sigma\mm Q^2$ is equivalent to knowing $\sigma_*\mm Q$, and knowing $\tau\mm Q^2$ is equivalent to knowing $\tau_*\mm Q$. One has
$b_R((j,k))=b_R((m,n))$ because $(j,R)=(m,R)$ and $(k,R)=(n,R)$. Next, $Qj= m-R\tau\equiv R(\tau_0-\tau)\equiv-RQ\tau_* \mm Q^2$, which implies $j\equiv -R\tau_*\mm Q$, so that $(j,Q)=(\tau_*,Q)$; in the same way, $Qk=n+R\sigma\equiv R(-\sigma_0+\sigma)\equiv RQ\sigma_*\mm Q^2$, so that $k\equiv R\sigma_*\mm Q$ and $(k,Q)=(\sigma_*,Q)$. Finally, $b_Q((j,k))=b_Q((\tau_*,\sigma_*))$.\\

We have just found that $j\equiv -R\tau_*\mm Q$ and $k\equiv R\sigma_*\mm Q$, so that
\begin{multline}
k\tau+j\sigma\equiv (R\sigma_*)\tau-(R\tau_*)\sigma\equiv R(\sigma_*\tau-\tau_*\sigma)\\
\equiv \sigma_*(m-Qj)+\tau_*(n-Qk)\equiv m\sigma_*+n\tau_* \mm Q.
\end{multline}
When $m,n$ have been fixed, one can thus substitute for the sum over the classes $\sigma,\tau$ mod $Q^2$ a sum over $\sigma_*,\tau_*$ regarded as classes mod $Q$. Starting from (\ref{813}), one ends up with the expression
\begin{multline}\label{816}
\frac{1}{Q^2} \sum_{m,n\in \Z} \sum_{\sigma_*,\tau_*\mm Q} b_R(m,n)\,b_Q(\tau_*,\sigma_*)\,\exp\l(2i\pi\frac{R}{Q}\l(m\sigma_*+n\tau_*\r)\r)\\
\times\,\intR \overline{w}\l(\frac{m}{Q}+t\r) w\l(\frac{m}{Q}-t\r) \exp\l(2i\pi\frac{nt}{Q}\r) dt.
\end{multline}\\

Now, one has
\begin{multline}
\frac{1}{Q} \sum_{\sigma_*,\tau_*\mm Q} b_Q((\tau_*,\sigma_*))\,\exp\l(2i\pi\frac{R}{Q}(m\sigma_*+n\tau_*)\r)\\
=\frac{1}{Q} \sum_{0\leq \sigma_*,\tau_*\mm Q} b_Q((R\tau_*,R\sigma_*))\,\exp\l(2i\pi\frac{R}{Q}(m\sigma_*+n\tau_*)\r)\\
\frac{1}{Q} \sum_{\sigma_*,\tau_*\mm Q} b_Q((\tau_*,\sigma_*))\,\exp\l(2i\pi\frac{m\sigma_*+n\tau_*}{Q}\r)
=\overset{\vee}{b}_Q((m,n)),
\end{multline}
the same as $\overset{\vee}{b}_Q((m,n))$. One obtains finally that the right-hand side of (\ref{810}) can be rewritten as
\begin{equation}
\frac{1}{Q}\sum_{m,n\in\Z} b_R((m,n))\,\overset{\vee}{b}_Q((m,n))\,\intR \overline{v}\l(\frac{m}{Q}+t\r) u\l(\frac{m}{Q}-t\r) \exp\l(2i\pi\frac{nt}{Q}\r) dt,
\end{equation}
which is the same as
\begin{align}
\frac{1}{Q}\,\Lan(\overset{\vee}{\mathfrak T}_N),\,(x,\xi)\mapsto \Wig(v,u)\l(\frac{x}{Q},\frac{\xi}{Q}\r)\Ran &=\Lan Q^{2i\pi\E}\overset{\vee}{\mathfrak T}_N),\,\Wig(v,u)\Ran\nonumber\\
&=\l(v\,\bigr|\,{\mathrm{Op}}\l(Q^{2i\pi\E}\overset{\vee}{\mathfrak T}_N)\r)u\r).
\end{align}\\

\end{proof}

\begin{lemma}\label{lem84}
Given a squarefree integer $Q\geq 1$, set, for $r,s\in \Z$,
\begin{equation}
\widehat{a}(r,s,Q)
=\frac{1}{Q}\sum_{j,k\,{\mathrm{mod}}\,Q} a^{\flat}((j,k,Q))\,\exp\l(2i\pi\frac{kr-js}{Q}\r).
\end{equation}
One has
\begin{equation}\label{821}
\widehat{a}(r,s,Q)=\mu(Q)\,\times\,a^{\flat}((r,s,Q)).
\end{equation}\\
\end{lemma}

\begin{proof}
Let us first show that if $Q=Q_1Q_2$, one has $\widehat{a}(r,s,Q)=
\widehat{a}(r,s,Q_1)\,\widehat{a}(r,s,Q_2)$: this will make it possible to reduce the
question to the case when $Q$ is prime. Fix two integers $R_1$ and
$R_2$ such that $R_1Q_1+R_2Q_2=1$ and, using the canonical group isomorphism $Z/Q\Z \sim
\Z/Q_1\Z\times \Z/Q_2\Z$, identify $j$ with a pair $j_1,j_2$ and $k$ with a pair $k_1,k_2$.
Then, $a^{\flat}((j,k,Q))=a^{\flat}((j_1,k_1,Q_1))\,a^{\flat}((j_2,k_2,Q_2))$ and
\begin{align}\label{822}
\exp\l(2i\pi\frac{kr-js}{Q}\r)&=\exp\l(2i\pi\frac{(kr-js)R_2}{Q_1}\r)
\exp\l(2i\pi\frac{(kr-js)R_1}{Q_2}\r)\nonumber\\
&=\exp\l(2i\pi\frac{(k_1r-j_1s)R_2}{Q_1}\r)\exp\l(2i\pi\frac{(k_2r-j_2s)R_1}{Q_2}\r).
\end{align}
Since $(R_2,Q_1)=1$, the pair $j_1R_2,k_1R_2$ also runs through the set of pairs of integers
mod $Q_1$ as the pair $j_1,k_1$ does; a similar remark concerns the second factor, and the
equation $\widehat{a}(r,s,Q)=\widehat{a}(r,s,Q_1)\,\widehat{a}(r,s,Q_2)$
follows.\\

If $p$ is prime, one has $a^{\flat}((j,k,p))=1$ unless both $j$ and $k$ are zero mod $p$, in
which case it is $1-p$, so that, $z$ being a primitive $p$th root of unity,
\begin{equation}
\widehat{a}(r,s,p)=\frac{1}{p}\sum_{j,k\,{\mathrm{mod}}\,p} z^{kr-js}-1.
\end{equation}
If $p\notdiv r$, one has $\sum_{k\,{\mathrm{mod}}\,p} z^{kr}=0$ and the first sum reduces to
$0$, the same being true if $p\notdiv s$: then, $\widehat{a}(r,s,p)=-1$. If both $r$ and $s$ are divisible by $p$, one has $\widehat{a}(r,s,p)=p-1=-(1-p)$. In all cases, $\widehat{a}(r,s,p)=-
a^{\flat}((r,s,p))$, which is the desired formula.\\
\end{proof}

The following is a proof of (\ref{82}) completely independent of the first one, and of Theorem \ref(A.Theorem \,6.1).\\

{\sc{A second proof of (\ref{82})}}\\

Recall (\ref{810}),
\begin{equation}
\l(v\,\big|\,\Psi\l(Q^{2i\pi\E}\,\overset{\vee}{\mathfrak T}_N\r)\,u\r)
=\l(v\,\big|\,\Psi\l({\mathfrak T}_N\r)\,\widetilde{u}\r),
\end{equation}
and use, a consequence of (\ref{87}), (\ref{89}) and Lemma \ref{lem84}, the fact that
$\overset{\vee}{\mathfrak T}_N=\mu(Q)\,{\mathfrak T}_N$. The equation (\ref{82}) follows.\\

But applying this formula together with Lemma \ref{lem84} leads to rather unpleasant (not published, though leading to the correct result) calculations, as we experienced. Instead, we shall give arithmetic the priority, starting from a manageable expression of the map $n\mapsto \overset{\vee}{n}$. Recall  (\ref{417}) that if $v,\,u$ are compactly supported and the algebraic sum of their supports is contained in some interval $[0,2\beta]$, and if one interests oneself in $\l(v\,\big|\,\Psi\l(Q^{2i\pi\E}{\mathfrak T}_{\frac{\infty}{2}}\r) u\r)$, one can replace ${\mathfrak T}_{\frac{\infty}{2}}$ by ${\mathfrak T}_N$ provided that $N=RQ$ is divisible by all odd primes $<\beta Q$. With this in mind, the following lemma will make it possible to assume without loss of generality that $R\equiv 1\mm 2Q^2$.\\

\begin{lemma}\label{lem82}
Let $Q$ be a squarefree odd positive integer and let $\beta>0$ be given. There exists $R>0$, with $N=RQ$ squarefree odd divisible by all odd primes $<\beta Q$, such that $R\equiv 1\mm 2Q^2$.\\
\end{lemma}

\begin{proof}
Choose $R_1$ positive, odd and squarefree, relatively prime to $Q$, divisible by all odd primes $<\beta Q$ relatively prime to $Q$, and $\overline{R}_1$ such that $\overline{R}_1R_1\equiv 1\mm 2Q^2$ and $\overline{R}_1\equiv 1\mm R_1$. Since $[R_1,2Q]=1$, there exists $x\in \Z$ such that $x\equiv 1 \mm R_1$ and $x\equiv \overline{R}_1\mm 2Q^2$. Choosing (Dirichlet's theorem) a prime $r$ such that $r\equiv x\mm 2R_1Q^2$, the number $R=R_1r$ satisfies the desired condition.\\
\end{proof}

With such a choice of $R$, we can make the map $n\mapsto \overset{\vee}{n}$ from $\Z/(2N^2)\Z$ to
$\Z/(2N^2)\Z$ explicit. Indeed, the solution of a pair of congruences $x\equiv\lambda \mm R^2,\,x\equiv \mu\mm{2Q^2}$ is given as
$x\equiv (1-R^2)\lambda +\mu R^2\mm 2N^2$. In particular, taking $\lambda=n,\,\mu=-n$, one obtains $\overset{\vee}{n}\equiv n(1-2R^2)\mm 2N^2$. Then, defining for instance $\widetilde{u}(y)=
u(y(1-2R^2))$, one has indeed $\l(\theta_N\,\widetilde{u}\r)(n)=(\theta_Nu)(\overset{\vee}{n})$. Of course, the support of $\widetilde{u}$ is not the same as that of $u$, but no support conditions are necessary for (\ref{82}) to hold.\\

We can now make a quick partial verification of the identity (\ref{76}), starting with the easy case for which $Q=1$.\\

\begin{lemma}\label{lem83}
Let $v,u$ be two functions in ${\mathcal S}(\R)$. One has for every squarefree integer $N$ the identity
\begin{equation}\label{825}
\l(v\,\big|\,\Psi\l({\mathfrak T}_N\r) u\r)=\sum_{T|N} \mu(T)\sum_{j,k\in \Z} \overline{v}\l(Tj+\frac{k}{T}\r) u\l(Tj-\frac{k}{T}\r).
\end{equation}
\end{lemma}

\begin{proof}
Together with the operator $2i\pi\E$, let us introduce the operator $2i\pi\E^{\natural}=r\,\frac{\partial}{\partial r}-s\,\frac{\partial}{\partial s}$ when the coordinates $(r,s)$ are used on $\R^2$. One has if ${\mathcal F}_2^{-1}$ denotes the inverse Fourier transformation with respect to the second variable ${\mathcal F}_2^{-1}\l[(2i\pi\E)\,{\mathfrak S}\r]=(2i\pi\E^{\natural})\,
{\mathcal F}_2^{-1}{\mathfrak S}$ for every tempered distribution ${\mathfrak S}$.
From the relation (\ref{49}) between ${\mathfrak T}_N$ and the Dirac comb, and Poisson's formula, one obtains
\begin{equation}\label{826}
{\mathcal F}_2^{-1}{\mathfrak T}_N= \prod_{p|N}\l(1-p^{-2i\pi\E^{\natural}}\r){\mathcal F}_2^{-1}{\mathcal Dir}= \sum_{T|N} \mu(T)\,T^{-2i\pi\E^{\natural}}{\mathcal Dir},
\end{equation}
explicitly
\begin{equation}
\l({\mathcal F}_2^{-1}{\mathfrak T}_N\r)(r,\,s)=
\sum_{T|N} \mu(T)\,\sum_{j,k\in\Z} \delta\l(\frac{r}{T}-j\r)\delta(Ts-k).
\end{equation}
The integral kernel of the operator $\Psi\l({\mathfrak T}_N\r)$ is (\ref{22})
\begin{align}\label{828}
K(x,\,y)&=\hf\l({\mathcal F}_2^{-1}{\mathfrak T}_N\r)\l(\frac{x+y}{2},\,\frac{x-y}{2}\r)\nonumber\\
&=\sum_{T|N} \mu(T)\,\sum_{j,k\in\Z}\delta\l(x-Tj-\frac{k}{T}\r)
\delta\l(y-Tj+\frac{k}{T}\r).
\end{align}
The equation (\ref{825}) follows.\\
\end{proof}

We now take advantage of (\ref{82}) and Lemma \ref{lem82}, after the proof of which we have seen that one could choose $\widetilde{u}(y)=u(y(1-2R^2)+2aN^2)$, with any $a\in \Z$, to obtain a quick verification of the main feature of (\ref{76}). As a particular case of the conditions in Theorem \ref{theo72}, we assume that the support of $v$ is contained in $[2,\sqrt{8}]$ and that of $u$ in $[0,1]$. From (\ref{825}), and (\ref{82}), one obtains
\begin{multline}
\l(v\,\big|\,\Psi\l(Q^{2i\pi\E}{\mathfrak T}_N\r) u\r)=\mu(Q)\sum_{T|N} \mu(T)\sum_{j,k\in \Z} \overline{v}\l(Tj+\frac{k}{T}\r) \widetilde{u}\l(Tj-\frac{k}{T}\r)\\
=\mu(Q)\sum_{T|N} \mu(T)\sum_{j,k\in \Z} \overline{v}\l(Tj+\frac{k}{T}\r) u\l(\l(Tj-\frac{k}{T}\r)(1-2R^2)+2aN^2\r).
\end{multline}
Let $x$ and $y$ be the arguments of $\overline{v}$ and $u$ in the last expression. Note that that of $u$ is truly defined mod $2N^2$, while no such proviso is made about $v$. The argument of $u$ here lies in $[0,1]$ for at most one value of $a$, which we choose so as to have
$x^2-y^2=4$, not only $x^2-y^2\equiv 4\mm 2N^2$. Then,

\begin{multline}
1=\frac{x^2-y^2}{4}=\frac{x-y}{2}\,\times\,\frac{x+y}{2}\\
=\hf\l[Tj+\frac{k}{T}-\l(Tj-\frac{k}{T}\r)(1-2R^2)-2aN^2\r]\\
\times\,\hf \l[Tj+\frac{k}{T}+\l(Tj-\frac{k}{T}\r)(1-2R^2)+2aN^2\r]\\
=\l[\frac{k}{T}+R^2\l(Tj-\frac{k}{T}\r)-aN^2\r]\,\l[Tj-R^2\l(Tj-\frac{k}{T}\r)+aN^2\r].
\end{multline}\\

With $\alpha=\frac{x-y}{2},\,\beta=\frac{x+y}{2}$, one has, as a congruence mod $2N^2$,
\begin{equation}
Q\,\beta\equiv Q\l[(1-R^2)Tj+\frac{R^2k}{T}\r]\equiv (1-R^2)QT_j+\frac{NRk}{T},
\end{equation}
an integer since $T|N$. Writing $R^2=1+2\lambda Q^2$, so that
\begin{equation}
\alpha\equiv\frac{(1-R^2)k}{T}+R^2Tj=-\frac{2\lambda Q^2k}{T}+R^2Tj,
\end{equation}
one sees in the same way that $R\alpha\in \Z$.\\

Since $\alpha\beta\equiv 1$, the numbers $m\colon=Q\beta$ and $n\colon=\frac{R}{\beta}=R\alpha$ are integers. Setting $m=m_1m_2$ and $n=n_1n_2$ with $m_1,n_1|R$ and $m_2,n_2|Q$, the equation $mn=QR$ yields $m_1n_1=\pm R,\,m_2n_2=\pm Q$ and $\beta=\frac{R}{n}=\pm\frac{m_1n_1}{n_1n_2}-\pm \frac{m_1}{n_2}$. In other words, $\beta=\frac{R_1}{Q_2}$ with $R_1|R$ and $Q_2|Q$. We have obtained that the arguments of $\overline{v}$ and $u$ are $x=\beta+\alpha=\frac{R_1}{Q_2}+\frac{Q_2}{R_1}$
and $y=\beta-\alpha=\frac{R_1}{Q_2}-\frac{Q_2}{R_1}$. This is an equality, not just a congruence, since adding to $R_1$ a multiple of $2N^2$ would put $\beta$ out of the interval $[0,\,\hf+2^{\hf}]$ where it must lie to contribute a nonzero term.\\

This is only a quick verification of Theorem \ref{theo72} (under the condition $R\equiv 1\mm 2Q^2$), not a complete second proof since recovering the integers $T,j,k$ from the set $\{R,Q,R_1,Q_2\}$ looks like a complicated task. In particular, we have not verified, here, that the coefficient of $\overline{v}\l(\frac{R_1}{Q_2}+\frac{Q_2}{R_1}\r)\,u\l(\frac{R_1}{Q_2}-\frac{Q_2}{R_1}\r)$ is $\mu(R_1Q_1)=\mu(Q)\mu(R_1Q_2)$. Theorem \ref{theo72}, based on (\ref{64}), gives the coefficients in full. Knowing, as is the case, that this coefficient is $0$ or $\pm 1$ since the integer $a$ such that $\widetilde{u}(y)=u(y(1-2R^2)+2aN^2)$ is unique would suffice for the application in Theorem \ref{theo92} below.\\

\section{Two series of hermitian forms and their integral versions}

\begin{definition}\label{def91}
Let $\eps>0$ be fixed. Given $u\in {\mathcal S}(\R)$, we introduce for $Q=1,2,\dots$ the function
$u_Q(y)=Q^{\frac{\eps}{2}}\,u(Q^{\eps}y)$ (no need to denote it as $u_{Q^{\eps}}$, as $\eps$ is fixed).We associate to any pair $v,u$ of functions in ${\mathcal S}(\R)$ the functions
\begin{align}\label{91}
F_{\eps}^{\,\flat}(v,\,u;\,s)\colon &=\sum_{Q\in \Sqo} Q^{-s}\l(v\,\big|\,\Psi\l(Q^{2i\pi\E}{\mathfrak T}_{\frac{\infty}{2}}\r)u_Q\r),\nonumber\\
F_{\eps}(v,\,u;\,s)\colon &=\sum_{Q\in \Sqo} Q^{-s}\l(v\,\big|\,\Psi\l(Q^{2i\pi\E}{\mathfrak S}_{\frac{\infty}{2}}\r)u_Q\r).
\end{align}\\
\end{definition}

\begin{theorem}\label{theo92}
Denote as $F_0^{\,\flat}(v,\,u;\,s),\,F_0(v,\,u;\,s)$ the two series obtained from
$F_{\eps}^{\,\flat}(v,\,u;\,s),F_{\eps}(v,\,u;\,s)$ after one has replaced $u_Q$ by $u$.
Let $v,u\in C^{\infty}(\R)$, with $v$ supported in $[2,\sqrt{8}]$ and $u$ supported in $[0,1]$. The functions $F_{\eps}^{\,\flat}(v,\,u;\,s),F_{\eps}(v,\,u;\,s)$ are entire if $\eps>0$. For $\Re s$ large (for $\Re s>2$, as will be made precise after Theorem {\em\ref{theo95}\/} has been established), each of them converges as $\eps\to 0$ towards the corresponding function from the set $F_0^{\,\flat}(v,\,u;\,s),\,F_0(v,\,u;\,s)$.\\
\end{theorem}

\begin{proof}
First observe for the peace of the mind that $F_0^{\flat}(v,\,u;\,s)$ is the same as the function similarly denoted in (\ref{52}). The support conditions on $v,u$ are somewhat arbitrary: the only point is that
one has $0<x^2-y^2<8$ when $v(x)u(y)\neq 0$, which will make it possible, when needed, to use the simple integral expression (\ref{317}) of $(v\,\big|\,\Psi({\mathfrak E}_{-\nu})\,u)$.\\

We use Theorem \ref{theo72}. Under the given support assumptions, one can rewrite (\ref{76}) as
\begin{align}\label{92}
\l(v\,\big|\,\Psi(Q^{2i\pi\E}{\mathfrak T}^{\,\flat}_{\frac{\infty}{2}})\,u\r)&=\sum_{Q_1Q_2=Q} \mu(Q_1)\nonumber\\
&\sum_{\begin{array}{c} R_1\in \Sqo \\(R_1,Q)=1 \end{array}} \mu(R_1)\,\overline{v}\l(\frac{R_1}{Q_2}+\frac{Q_2}{R_1}\r)\,
u\l(\frac{R_1}{Q_2}-\frac{Q_2}{R_1}\r).
\end{align}
Indeed, given $Q$, $\l(v\,\big|\,\Psi(Q^{2i\pi\E}{\mathfrak T}_{\frac{\infty}{2}})\,u\r)$ coincides according to (\ref{417}) with $\l(v\,\big|\,\Psi(Q^{2i\pi\E}{\mathfrak T}_N)\,u\r)$ for $N$ odd divisible, for some $\beta>0$ depending on the supports of $v$ and $u$, by all odd primes $<\beta Q$:
we shall take $\beta=\frac{1+\sqrt{8}}{2}$ and fix $N$ accordingly. Now, given $R_1$ squarefree odd, the condition that $R_1$ divides some number $R$ with the property that $N=RQ$ is squarefree odd implies the condition $(R_1,Q)=1$. In the reverse direction, the condition $0<\frac{R_1}{Q_2}-\frac{Q_2}{R_1}<1$ implies $R_1<\frac{1+\sqrt{5}}{2}\,Q$. This proves, since $R_1$ is squarefree, that $R_1|N$ or,
if $(R_1,Q)=1$, that $R_1|R$.\\

Then, for $\Re s>2$,
\begin{multline}\label{93}
F_{\eps}^{\,\flat}(v,\,u;\,s)\colon =\sum_{Q\in \Sqo} Q^{-s}\l(v\,\big|\,\Psi\l(Q^{2i\pi\E}{\mathfrak T}_{\frac{\infty}{2}}\r)\,u_Q\r)=\sum_{Q\in \Sqo} Q^{-s+\frac{\eps}{2}}\\
\sum_{Q_1Q_2=Q} \mu(Q_1)
\sum_{\begin{array}{c} R_1\in \Sqo \\(R_1,Q)=1 \end{array}} \mu(R_1)\,\overline{v}\l(\frac{R_1}{Q_2}+\frac{Q_2}{R_1}\r)\,
u\l(Q^{\eps}\l(\frac{R_1}{Q_2}-\frac{Q_2}{R_1}\r)\r).
\end{multline}
For all nonzero terms of this series, one has for some absolute constant $C$
\begin{equation}\label{94}
\l|\frac{R_1}{Q_2}-\frac{Q_2}{R_1}\r|\leq Q^{-\eps},\quad
\l|\frac{R_1}{Q_2}-1\r|\leq C\,Q^{-\eps},\quad
|R_1-Q_2|\leq C\,Q_2Q^{-\eps}.
\end{equation}
The number of available $R_1$ is at most $C\,Q_2Q^{-\eps}$. On the other hand,
\begin{equation}
\l|\frac{R_1}{Q_2}+\frac{Q_2}{R_1}-2\r|\leq \l|\frac{R_1}{Q_2}-1\r|+\l|\frac{Q_2}{R_1}-1\r|\leq C\,Q^{-\eps}.
\end{equation}
Since $v(x)$ is flat at $x=2$, one has for every $M\geq 1$ and some constant $C_M$
\begin{multline}\label{96}
\l|\l(v\,\big|\,\Psi\l(Q^{2i\pi\E}{\mathfrak T}_N\r)\,u_Q\r)\r|\leq CC_M\,Q^{\frac{\eps}{2}}\sum_{Q_1Q_2=Q} (Q_2Q^{-\eps})\,Q^{-M\eps}\\
\leq CC_MC_2\,Q^{\frac{\eps}{2}+1-(M+1)\eps+\eps'}
=CC_MC_2\,Q^{1-(M+\hf)\eps+\eps'},
\end{multline}
where $C_2\,Q^{\eps'}$ is a bound ($\eps'$ can be taken arbitrarily small) for the number of divisors of $Q$.\\

Using (8.1) and (8.6), one has for the above choice of $N$ (use the fact that the support of $u_Q$ is contained in that of $u$)
\begin{multline}
\sum_{Q\in \Sqo} \l|\,Q^{-s}\l(v\,\,\big|\,\Psi\l(Q^{2i\pi\E}{\mathfrak T}_{\frac{\infty}{2}}\r) u_Q\r)\r|=\sum_{Q\in \Sqo} \l|\,Q^{-s}\l(v\,\,\big|\,\Psi\l(Q^{2i\pi\E}{\mathfrak T}_N\r) u_Q\r)\r|\\
\leq \sum_{Q\in \Sqo} Q^{-\Re s}\,\times\,CC_MC_2Q^{1-(M+\hf)\eps+\eps'}<\infty,
\end{multline}
taking $M$ large enough so that $-\Re s-(M+\hf)\eps+\eps'<-2$.\\

The theorem follows, so far as the series $F_{\eps}^{\,\flat}(v,\,u;\,s)$ is concerned. Changing the series for $F_{\eps}^{\,\flat}(v,\,u;\,s)$ to that for $F_{\eps}(v,\,u;\,s)$ just demands replacing factors $\mu(r)$, in all their occurrences, by their absolute values.\\
\end{proof}

We wish now to compute $F_{\eps}^{\,\flat}(v,\,u;\,s)$ by analytic methods.
The novelty, in comparison to the analysis made in Section 5, is that we have replaced $u$ by $u_Q$: in just the same way as the distribution ${\mathfrak T}_{\frac{\infty}{2}}$ was decomposed into Eisenstein distributions, it is $u$ that we must now decompose into homogeneous components.
Let us recall (Mellin transformation, or Fourier transformation up to a change of variable) that functions on the line decompose into generalized eigenfunctions of the self-adjoint operator $i\l(\hf+y\frac{d}{dy}\r)$, according to the decomposition (analogous to (\ref{310}))
\begin{equation}\label{97}
u=\frac{1}{i}\int_{\Re \mu=0} u^{\mu}\,d\mu,
\end{equation}
with
\begin{equation}\label{98}
u^{\mu}(y)=\frac{1}{2\pi}\int_0^{\infty} \theta^{\mu-\hf}u(\theta y)\,d\theta.
\end{equation}
The function $u^{\mu}$ is homogeneous of degree $-\hf-\mu$. Note that $\mu$ is in the superscript position, in order not to confuse $u^{\mu}$ with a case of $u_Q$. A justification of (\ref{97}),\,(\ref{98}) is as follows.\\

Given $u=u(y)$ in ${\mathcal S}(\R)$ and $y\neq 0$, consider the function $f(t)=e^{\pi t}u\l(e^{2\pi t}y\r)$. The functions $f$ and $\widehat{f}$ are integrable. Define
\begin{equation}
u^{i\lambda}(y)\colon =\widehat{f}(-\lambda)=\frac{1}{2\pi}\int_0^{\infty} \theta^{i\lambda-\hf}\,u(\theta y)\,d\theta.
\end{equation}
Then, $u(y)=f(0)=\intR \widehat{f}(-\lambda)\,d\lambda=\frac{1}{i}\int_{\Re\mu=0} u^{\mu}\,d\mu$.\\

\begin{proposition}\label{prop93}
Let $v,\,u \in C^{\infty}(\R)$ satisfy the support conditions in Theorem {\em\ref{theo92}\/}. For every $\nu \in \C$, the function
\begin{equation}\label{910}
\mu\mapsto \Phi(v,\,u;\,\nu,\,\mu)\colon
=\int_0^{\infty} t^{\nu-1}\,\overline{v}(t+t^{-1})\,u^{\mu}(t-t^{-1})\,dt,
\end{equation}
initially defined for $\Re\mu>-\hf$, extends as an analytic function, rapidly decreasing in vertical strips. One has
\begin{equation}\label{911}
(v\,\big|\,\Psi({\mathfrak E}_{-\nu})\,u)=\frac{1}{i}\,\int_{\Re\mu=0}
\Phi(v,\,u;\,\nu,\,\mu)\,\,d\mu.
\end{equation}\\
\end{proposition}

\begin{proof}
When $\overline{v}(t+t^{-1})\neq 0$, $t$ is bounded and bounded away from zero.
The last factor of the integrand of (\ref{910}) has a singularity at $t=1$, taken care of by the fact that $(t-t^{-1})^2\geq C^{-1}\,(t+t^{-1}-2)$ for $t^{\pm 1}$ bounded, while $v=v(x)$ is flat at $x=2$.
Observe at this point the decisive advantage of using a pseudodifferential operator: after having replaced $u(t-t^{-1})$ by $u^{\mu}(t-t^{-1})$, we can no longer eliminate the singularity at $t=1$ by means of a support condition relative to $u$. But we can take advantage of the factor $v$ for the same effect: we might even, so as to dispense with the immediate inequality just used, assume that the support of $v$ is contained in a compact subset of $]2,\sqrt{8}]$.\\

Powers of $(1+|\mu|)^{-1}$ can be gained with the help of repeated integrations by parts associated to the identity
\begin{equation}\label{912}
(-\hf-\mu)\,u^{\mu}(t-t^{-1})=(t+t^{-1})^{-1}\,t\frac{d}{dt}\l(u^{\mu}(t-t^{-1})\r),
\end{equation}
to be combined with $t\frac{d}{dt}(t\pm t^{-1})=t\mp t^{-1}$.\\

Let us prove (\ref{911}). Using (\ref{97}), one has
\begin{equation}\label{913}
\frac{1}{i}\,\int_{\Re\mu=0}\Phi(v,\,u;\,\nu,\,\mu)\,\,d\mu
=\int_0^{\infty} t^{\nu-1}\,\overline{v}(t+t^{-1})\,u(t-t^{-1})\,dt.
\end{equation}
The right-hand side is the same as $(v\,\big|\,\Psi({\mathfrak E}_{-\nu})\,u)$ according to Theorem \ref{theo34}.\\

The equations (\ref{97}) and (\ref{911}) do not imply, however, that $\Phi(v,\,u;\,\nu,\,\mu)$ is the same as $(v\,\big|\,\Psi({\mathfrak E}_{-\nu})\,u^{\mu})$, because the pair $v,u^{\mu}$ does not satisfy the support conditions in the second part of Theorem \ref{theo34}. Actually, as seen from the proof of Theorem \ref{theo34},
\begin{equation}
\Phi(v,\,u;\,\nu,\,\mu)=(v\,\big|\,\Psi({\mathfrak S}_1)\,u^{\mu}),
\end{equation}
where ${\mathfrak S}_1(x,\,\xi)=|x|^{\nu-1}\exp\l(-\frac{2i\pi\xi}{x}\r)$ is the term corresponding to the choice $r=1$ in the Fourier expansion (\ref{318}) of ${\mathfrak E}_{-\nu}$. Such a reduction has been made possible by the demands made on the supports of $v,u$\\
\end{proof}

We also need to prove that, with a loss tempered by a power of $|\mu|$, the function $\nu\mapsto \Phi(v,\,u;\,\nu,\,\mu)$ is integrable on lines $\Re \nu=c$ with $c>1$. The following is more precise than what would truly be needed.\\

\begin{proposition}\label{prop94}
Let $v=v(x)$ and $u=u(y)$ be two functions satisfying the conditions in Theorem {\em\ref{theo92}\/}. Defining the operators, to be applied to $v$, such that
\begin{equation}\label{915}
D_{-1}^{\mu}v=v'',\quad D_0^{\mu}=-2\mu\l(xv'+\frac{v}{2}\r),\qquad
D_1^{\mu}v=\l(\hf+\overline{\mu}\r)\l(\frac{3}{2}+\overline{\mu}\r) x^2v,
\end{equation}
one has if $\Re\mu<\hf$
\begin{equation}\label{916}
\l(\hf+\nu^2\r)\Phi(v,\,u;\,\nu,\,\mu)=\sum_{j=-1,0,1} \Phi\l(D_j^{\mu}v,\,\,y^{-2j}u;\,\,\nu,\,\mu+2j\r).
\end{equation}
As a function of $\nu$ on any line $\Re \nu=c$ with $c>1$, the function $\Phi(v,\,u;\,\nu,\,\mu)$
is a ${\mathrm{O}}((1+|\Im\nu|)^{-2})$, with a loss of uniformity relative to $\mu$ bounded by $(1+|\mu|)^2$.\\
\end{proposition}

\begin{proof}
The equation (\ref{910}) and the identity $\nu^2\,t^{\nu}=(t\frac{d}{dt})^2t^{\nu}$ give (noting that the operator $t\frac{d}{dt}$ is the negative of its transpose if one uses the measure $\frac{dt}{t}$)
\begin{equation}
\nu^2\,\Phi(v,\,u;\,\nu,\,\mu)\colon
=\int_0^{\infty} t^{\nu-1}\,(t\frac{d}{dt})^2\,\l(\overline{v}(t+t^{-1})\,u^{\mu}(t-t^{-1})\r)\,dt.
\end{equation}
To facilitate the calculations which follow, observe, setting $|s|_1^{\alpha}=|s|^{\alpha}{\mathrm{sign}} s$, that $\frac{d}{ds}\,|s|^{\alpha}=\alpha\,|s|_1^{\alpha-1}$ and
$\frac{d}{ds}\,|s|_1^{\alpha}=\alpha\,|s|^{\alpha-1}$, finally $t\frac{d}{dt}\,v(t+t^{-1})=(t-t^{-1})\,v'(t+t^{-1})$ and
$t\frac{d}{dt}\,u(t-t^{-1})=(t+t^{-1})\,u'(t-t^{-1})$.\\

One has
\begin{multline}
t\frac{d}{dt}\l[\overline{v}(t+t^{-1})\,|t-t^{-1}|^{-\mu-\hf}\r]\\
=\overline{v'}(t+t^{-1})\,|t-t^{-1}|_1^{-\mu+\hf}-(\mu+\hf)\,\overline{v}(t+t^{-1})\,
(t+t^{-1})\,|t-t^{-1}|_1^{-\mu-\frac{3}{2}}.
\end{multline}
Next,
\begin{multline}
(t\frac{d}{dt})^2\l[\overline{v}(t+t^{-1})\,|t-t^{-1}|^{-\mu-\hf}\r]\\
=\overline{v''}(t+t^{-1})\,|t-t^{-1}|^{-\mu+\frac{3}{2}}
-2\mu\,(t+t^{-1})\,\overline{v'}(t+t^{-1})\,|t-t^{-1}|^{-\mu-\hf}\\
-(\mu+\hf)\,\overline{v}(t+t^{-1})\,|t-t^{-1}|^{-\mu-\hf}\\
+(\mu+\hf)(\mu+\frac{3}{2})\,(t+t^{-1})^2\,\overline{v}(t+t^{-1})\,
|t-t^{-1}|^{-\mu-\frac{5}{2}}.
\end{multline}\\

Now, one has $y^2u^{\mu}=\l(y^2u\r)^{\mu-2},\,\,y^{-2}u^{\mu}=\l(y^{-2}u\r)^{\mu+2}$.
If $u$ is even, so that $u^{\mu}(y)$ is a multiple of $|y|^{-\mu-\hf}$, using these identities with $y=t-t^{-1}$ leads to the identity (\ref{916}). Just exchanging the signed and unsigned versions of the power function gives the same result if $u$ is odd. The loss by a factor $(1+|\mu|)^2$ is insignificant in view of Proposition \ref{prop93}.\\
\end{proof}

In the next theorem, we denote a pair of functions useful in all that follows as $f^{\,\flat},f$ since they correspond to the pair ${\mathfrak T},\,{\mathfrak S}$ or, ultimately, to the pair $a^{\flat},a$. Incidentally, our choice of not denoting ${\mathfrak T}$ as ${\mathfrak S}^{\,\flat}$ was made just so as to avoid visually unpleasant notation such as ${\mathfrak S}^{\,\flat}_{\frac{\infty}{2}}$.\\

\begin{theorem}\label{theo95}
Set
\begin{equation}\label{ffl}
f^{\,\flat}(\nu)=\frac{\l(1-2^{-\nu}\r)^{-1}}{\zeta(\nu)},\qquad f(\nu)=(1+2^{-\nu})^{-1}\frac{\zeta(\nu)}{\zeta(2\nu)}.
\end{equation}
Let $v,u$ satisfy the conditions of Theorem {\em\ref{theo92}\/}. With
\begin{equation}\label{921}
H_{\eps}(v,\,u;\,\nu;\,s)=\frac{1}{i}\,\int_{\Re\mu=0} f(s-\nu+\eps\mu)\,\Phi(v,\,u;\,\nu,\,\mu)\,d\mu,
\end{equation}
one has if $c>1$ and $\Re s$ is large the identities
\begin{equation}\label{920}
F_{\eps}^{\,\flat}(v,\,u;\,s)=\frac{1}{2i\pi}\int_{\Re\nu=c} f^{\,\flat}(\nu)\,\,H_{\eps}(v,\,u;\,\nu;\,s)\,d\nu
\end{equation}
and
\begin{equation}\label{921.5}
F_{\eps}(v,\,u;\,s)=\frac{1}{2i\pi}\int_{\Re\nu=c} f(\nu)\,\,H_{\eps}(v,\,u;\,\nu;\,s)\,d\nu.
\end{equation}
\end{theorem}

\begin{proof}
Using (\ref{91}) and (\ref{417}), one has for $c>1$
\begin{equation}\label{922}
F_{\eps}^{\,\flat}(v,\,u;\,s)=\frac{1}{2i\pi} \sum_{Q\in\Sqo} Q^{-s}\,
\int_{\Re\nu=c} f^{\,\flat}(\nu) \,Q^{\nu}\,
\l(v\,\big|\,\Psi\l({\mathfrak E}_{-\nu}\r)u_Q\r)\,d\nu.
\end{equation}
One has $(u_Q)^{\mu}=(u^{\mu})_Q=Q^{-\eps\mu}u^{\mu}$, and, according to (\ref{910}) and (\ref{911}),
\begin{equation}\label{923}
\l(v\,\big|\,\Psi\l({\mathfrak E}_{-\nu}\r)u_Q\r)=\frac{1}{i}\int_{\Re\mu=0}
Q^{-\eps\mu}\Phi(v,\,u;\,\nu,\,\mu)\,d\mu.
\end{equation}
Recall (Proposition \ref{prop93}) that $\Phi(v,\,u;\,\nu,\,\mu)$ is a rapidly decreasing function of $\mu$ in vertical strips.\\

To insert this equation into (\ref{922}), we use Proposition \ref{prop94}, which provides the $d\nu$-summability, at the price of losing at most the factor $(1+|\mu|)^2$. We obtain if $c>1$ and $\Re s$ is large enough
\begin{equation}
F_{\eps}^{\,\flat}(v,\,u;\,s)=-\frac{1}{2\pi} \sum_{Q\in \Sqo}
\int_{\Re\nu=c} Q^{-s+\nu}\, \,d\nu
\int_{\Re\mu=0}
Q^{-\eps\mu}\,\Phi(v,\,u;\,\nu,\,\mu)\,d\mu
\end{equation}
or, using (\ref{54}),
\begin{equation}\label{925}
F_{\eps}^{\,\flat}(v,\,u;\,s)=-\frac{1}{2\pi} \int_{\Re\nu=c} f^{\,\flat}(\nu) \,d\nu
\int_{\Re\mu=0} f(s-\nu+\eps\mu)\,\Phi(v,\,u;\,\nu,\,\mu)\,d\mu,
\end{equation}
with
\begin{equation}\label{929b}
f(s-\nu+\eps\mu)=\sum_{Q\in\Sqo}Q^{-s+\nu-\eps\mu}=
\l(1+2^{-s+\nu-\eps\mu}\r)^{-1}\,\frac{\zeta(s-\nu+\eps\mu)}
{\zeta(2(s-\nu+\eps\mu))},
\end{equation}
where the last equation is a consequence of (\ref{410}).\\

The proof of (\ref{921.5}) is identical to that of (\ref{920}), only replacing $f^{\,\flat}(\nu)$ by $f(\nu)$. Please note that the factor $f(s-\nu+\eps\mu)$ is the same in both cases: it originates from our having made the $Q$-summation over the set of squarefree odd integers. The origin of the factor $f^{\,\flat}(\nu)$ or $f(\nu)$ is quite different: it goes back to the spectral decomposition of the distribution ${\mathfrak T}_{\infty}$ or  ${\mathfrak S}_{\infty}$.\\
\end{proof}

A resolvent of the self-adjoint operator $i\l(y\frac{d}{dy}+\hf\r)$ in $L^2(\R)$ is given by the equations
\begin{equation}
\l[\l(y\frac{d}{dy}+\hf+\mu\r)^{-1}\,u\r](y)=\begin{cases}
\int_0^1 \theta^{\mu+\hf}\,u(\theta y)\,\frac{d\theta}{\theta}\qquad {\mathrm{if}}\,\,\Re\mu>0,\\
-\int_1^{\infty} \theta^{\mu+\hf}\,u(\theta y)\,\frac{d\theta}{\theta}\qquad {\mathrm{if}}\,\,\Re\mu<0.
\end{cases}
\end{equation}

Let $u\in {\mathcal S}(\R)$ be flat at $0$ and let $\mu\in\C$. Starting from (\ref{98}), setting $\delta=y\frac{d}{dy}$ and using the integration by parts associated to the identity $\theta\frac{d}{d\theta}\theta^{\mu+\hf}=(\mu+\hf)\,\theta^{\mu+\hf}$, one obtains for $n=1,2,\dots$
and $\mu\neq -\hf$ the identity
\begin{equation}\label{927}
u^{\mu}(y)=\frac{(-1)^n}{(\mu+\hf)^n}\,.\,\frac{1}{2\pi}\,\int_0^{\infty}\theta^{\mu+\hf}\,
(\delta^n\,u)(\theta y)\,\frac{d\theta}{\theta},\qquad y\neq 0.
\end{equation}\\

\begin{proposition}\label{prop96}
With $u\in {\mathcal S}(\R)$ supported in $[0,1]$, define for $n=1,2,\dots$, for $y$ in some compact subinterval of $\big]0,1]$ and $\mu\neq -\hf$ the functions
\begin{align}\label{928}
u^{\mu,n}_+(y)&=\frac{(-1)^{n+1}}{(\mu+\hf)^n}\,.\,\frac{1}{2\pi}\,\int_0^1 \theta^{\mu+\hf}\,(\delta^n\,u)(\theta y)\,\frac{d\theta}{\theta},\nonumber\\
u^{\mu,n}_-(y)&=\frac{(-1)^n}{(\mu+\hf)^n}\,.\,\frac{1}{2\pi}\,\int_1^{\infty} \theta^{\mu+\hf}\,(\delta^n\,u)(\theta y)\,\frac{d\theta}{\theta},
\end{align}
so that
\begin{equation}
\frac{1}{i}\int_{\Re\mu=0}\l(u^{\mu,n}_++u^{\mu,n}_-\r) d\mu=\frac{1}{i}\int_{\Re\mu=0} u^{\mu}\,d\mu=u.
\end{equation}
As a function of $\mu$, $u^{\mu,n}_-(y)$ is regular for $\mu\neq -\hf$ and has at this point a pole of order $\leq n$, of order $n$ if $(\delta^{n-1}u)(y)\neq 0$. Recall that the function $\mu\mapsto u^{\mu}(y)=u^{\mu,n}_+(y)+u^{\mu,n}_-(y)$ is entire.\\

Keeping away from the pole, say by the condition $|\Im\mu|\geq 1$, the functions $(1+|\mu|)^{n+1}\,u_{\pm}^{\mu,n}(y)$ are uniformly bounded in the remaining part of any strip $\{\mu\colon \alpha\leq \Re\mu\leq \beta\}$. Bounds are uniform with respect to $y$ as well as long as $y$ remains bounded away from zero. These bounds are valid too, with the same $n$, for derivatives
of any order of $u_{\pm}^{\mu,n}$ (or images of $u_{\pm}^{\mu,n}$ under powers of the operator $y\frac{d}{dy}$)\\
\end{proposition}

\begin{proof}
An integration by parts yields
\begin{align}\label{930}
u^{\mu,n}_-(y)&=\frac{1}{2\pi}\,\frac{(-1)^n}{(\mu+\hf)^n}\nonumber\\
&\times \l[-(\delta^{n-1}u)(y)-\int_1^{\infty} (\mu+\hf)\,\theta^{\mu+\hf}\,(\delta^{n-1}u)(\theta y)\,\frac{d\theta}{\theta}\r]\nonumber\\
&=\frac{1}{2\pi}\,\frac{(-1)^{n-1}}{(\mu+\hf)^n}\,(\delta^{n-1}u) (y)+u^{\mu,n-1}_-(y).
\end{align}
Pushing the integration by parts further, write
\begin{equation}\label{934}
u^{\mu,n}_-(y)=\frac{1}{2\pi}\,\frac{(-1)^{n-1}}{(\mu+\hf)^n}\,(\delta^{n-1}u)(y)+\dots
-\frac{1}{2\pi}\,\frac{1}{(\mu+\hf)^2}\,(\delta u)(y)+u^{\mu,1}_-(y),
\end{equation}
with
\begin{equation}\label{935}
u^{\mu,1}_-(y)=\frac{1}{2\pi}\l[\frac{u(y)}{\mu+\hf}+\int_1^{\infty}\theta^{\mu+\hf}
u(\theta y)\,\frac{d\theta}{\theta}\r].
\end{equation}
The first part of the proposition follows: recall that $y\neq 0$.\\

Then, we note that, when $(\delta^nu)(\theta y)\neq 0$, $\theta$ is bounded and bounded from below.
This gives almost the last statement, only with $n+1$ in $(1+|\mu|)^{n+1}$ replaced by $n$. We then read (\ref{930}) backwards, as
\begin{equation}\label{936}
u^{\mu,n}_-(y)=u^{\mu,n+1}_-(y)+
\frac{1}{2\pi}\,\frac{(-1)^{n+1}}{(\mu+\hf)^{n+1}}\,(\delta^nu) (y),
\end{equation}
gaining with the help of (\ref{928}) a factor $\frac{1}{\mu+\hf}$ in the estimate of $u^{\mu,n}_-(y)$ as $|\Im\mu|\to \infty$. In particular,
\begin{equation}\label{937}
u_-^{\mu,1}(y)=\frac{1}{2\pi}\,\frac{1}{(\mu+\hf)^2}\,\l[(\delta u)(y)+\int_1^{\infty} \theta^{\mu+\hf}(\delta^2 u)(\theta y)\,\frac{d\theta}{\theta}\r].
\end{equation}\\

The dependence on $y$ is obtained from an obvious change of variable in (\ref{928}). The last assertion is obtained from the equation $y\frac{d}{dy}\l(u_{\pm}^{\mu,n}\r)=(\mu+\hf)\,u_{\pm}^{\mu,n+1}$.\\

\end{proof}

\section{Proof of the Riemann hypothesis}

\begin{theorem}\label{theo102}
Let $v,\,u\in C^{\infty}(\R)$ satisfy the assumptions of Theorem {\em 9.2\/}, here recalled: $v$ is supported in $[2,\sqrt{8}]$ and $u$ in $[0,1]$. For $c>1$ and $2<\Re s<c+1$, one has $F_{\eps}^{\flat}(s)=E_{\eps}^{\flat}(s)-i\,G_{\eps}^{\flat}(s)$, with
\begin{equation}\label{104}
E_{\eps}^{\flat}(s)=-\frac{1}{2\pi}\int_{\Re\mu=0}\int_{\Re\nu=c} \frac{\l(1-2^{-\nu}\r)^{-1}}{\zeta(\nu)}\,f(s-\nu+\eps\mu)\,\Phi(v,\,u;\,\nu,\,\mu)\,d\mu\,d\nu
\end{equation}
and
\begin{equation}\label{105}
G_{\eps}^{\flat}(s)=\frac{4}{\pi^2}\,\int_{\Re \mu=0}\frac{\l(1-2^{-s+1-\eps\mu}\r)^{-1}}{\zeta(s-1+\eps\mu)}\,
\Phi(v,\,u;\,s-1+\eps\mu,\,\mu)\,d\mu.
\end{equation}
We recall {\em(8.20)\/} that $f(s)=\l(1+2^{-s}\r)^{-1}\frac{\zeta(s)}{\zeta(2s)}$ and that {\em(8.10)\/}, for $\Re\mu>-\hf$,  \begin{equation}
\Phi(v,\,u;\,\nu,\,\mu)
=\int_0^{\infty} t^{\nu-1}\,\overline{v}(t+t^{-1})\,u^{\mu}(t-t^{-1})\,dt.
\end{equation}\\

Set $\sigma_0={\mathrm{sup}}\,\{\Re \rho\colon \zeta(\rho)=0\}$. The function $G_{\eps}^{\flat}(s)$, as defined by {\em(\ref{105})\/} for $\Re (s-1)>\sigma_0$, extends analytically to the half-plane $\Re (s-1)>\frac{\sigma_0}{2}$, and the decomposition $F_{\eps}^{\flat}(s)=E_{\eps}^{\flat}(s)-i\,G_{\eps}^{\flat}(s)$ is valid for $c+\frac{\sigma_0}{2}<\Re s<c+1$.\\
\end{theorem}

\begin{proof}
We start from Theorem \ref{theo95}, here recalled,
\begin{equation}\label{106}
F_{\eps}^{\flat}(s)=-\frac{1}{2\pi} \int_{\Re\nu=c} \int_{\Re\mu=0} \frac{\l(1-2^{-\nu}\r)^{-1}}{\zeta(\nu)}\, f(s-\nu+\eps\mu)\,\Phi(v,\,u;\,\nu,\,\mu)\,d\nu\,d\mu,
\end{equation}
an identity valid if $c>1$ and $\Re s>c+1$. The double integral is absolutely convergent in view of Propositions \ref{prop93} and \ref{prop94}: an iteration of the second gives for the $\Phi$-function bounds by arbitrary powers with positive exponents of $\frac{1+|\mu|^2}{|\hf+\nu^2|}$, while the first provides the extra power of $(1+|\mu|)^{-1}$ needed.\\

We first make a change of the $d\nu$-contour, for $\mu$ fixed; in the next section, we shall make a change of the $d\mu$-contour present in (9.7) below.\\

The function $f(s-\nu+\eps\mu)$ is non-singular not only for $\Re(s-\nu+\eps\mu)>1$, but also for $\frac{\sigma_0}{2}<\Re(s-\nu+\eps\mu)<1$: the first (resp.\,second) constrain implies that the denominator (resp.\,numerator) of the fraction $\frac{\zeta(s-\nu+\eps\mu)}{\zeta(2(s-\nu+\eps\mu))}$ is nonzero (resp.\,nonsingular). The integral (\ref{106}) is thus convergent not only for $\Re s>c+1$, but also for $c+\frac{\sigma_0}{2}<\Re s<c+1$. The convergence is locally uniform not only relative to $s$ in this strip but relative to $\Re s$ as well.\\

To make the jump possible, we choose $c_1$ such that $1<c_1<c$ and observe that one can replace in (\ref{106}) the integral sign $\int_{\Re\nu=c}$ by $\int_{\Re\nu=c_1}$, enlarging the domain of validity of the new identity to the half-plane $\Re s>c_1+1$. In the case when $c_1+1<\Re s<c+1$, one has for $\Re\mu=0$
\begin{multline}\label{107b}
\frac{1}{2i\pi}\int_{\Re\nu=c_1} \frac{\l(1-2^{-\nu}\r)^{-1}}{\zeta(\nu)}\,f(s-\nu+\eps\mu)\,\Phi(v,\,u;\,\nu,\,\mu)\,d\nu\\
=\frac{1}{2i\pi}\int_{\Re\nu=c} \frac{\l(1-2^{-\nu}\r)^{-1}}{\zeta(\nu)}\,f(s-\nu+\eps\mu)\,\Phi(v,\,u;\,\nu,\,\mu)\,d\nu\\
-{\mathrm{Res}}_{\nu=s-1+\eps\mu}\l[\frac{\l(1-2^{-\nu}\r)^{-1}}{\zeta(\nu)}\,f(s-\nu+\eps\mu)\,
\Phi(v,\,u;\,\nu,\,\mu)\r].
\end{multline}
Let us check in detail the validity of this change of contour, drawing in one's mind a large rectangle the vertical sides of which are the lines $\Re\nu=c$ and $\Re\nu=c_1$, to be completed by two horizontal segments at heights $\pm a$ with $a$ large. For $a$ large, the point $\nu=s-1+\eps\mu$ is indeed a pole, and the only pole, of the integrand one comes across during the move. We have already verified the convergence of the integrals on the two (full) vertical lines, and Proposition \ref{prop94} shows that, given $\mu$, the integral on the two horizontal sides go to zero as $a\to \infty$.\\

The $d\nu$-integral which is the first term on the right-hand side is analytic in the domain $\{s\colon c+\frac{\sigma_0}{2}<\Re s<c+1\}$: so is its $d\mu$-integral $E_{\eps}^{\flat}(s)$, (\ref{912}) ensuring the $d\mu$-summability. On the other hand, the $d\mu$-integral of the left-hand side coincides with $F_{\eps}^{\flat}(s)$ for $\Re s>c_1+1$ according to Theorem \ref{theo95}, hence extends as an entire function (Theorem \ref{theo92}): recall again that this is a pseudodifferential-arithmetic fact, which could not be obtained by Cauchy-type analysis only.\\

The residue of the function $f(s)=\l(1+2^{-s}\r)^{-1}\,\frac{\zeta(s)}{\zeta(2s)}$ at $s =1$ is $\frac{4}{\pi^2}$. The second line on the right-hand side of (\ref{107b}) is thus
\begin{equation}\label{108}
\frac{4}{\pi^2}\,\frac{\l(1-2^{-s+1-\eps\mu}\r)^{-1}}{\zeta(s-1+\eps\mu)}\,
\Phi(v,\,u;\,s-1+\eps\mu,\,\mu)=\frac{4}{\pi^2}\,f^{\flat}(s-1+\eps\mu)\,
\Phi(v,\,u;\,s-1+\eps\mu,\,\mu)
\end{equation}
with $f^{\flat}(s)=\frac{(1-2^{-s})^{-1}}{\zeta(s)}$.\\

Whatever the pair $c_1,c$ such that $1<c_1<c$, one has by definition $G_{\eps}^{\flat}(s)=i(F_{\eps}^{\flat}(s)-E_{\eps}^{\flat}(s))$ if $c_1+1<\Re s<c+1$. The right-hand side of this identity is analytic for $c+\frac{\sigma_0}{2}<\Re s<c+1$. The function $G_{\eps}^{\flat}(s)$, in the definition of which $c>1$ is not present, thus extends analytically for $\Re s> 1+\frac{\sigma_0}{2}$: then, forget about the upper bound for $\Re s$.\\

The integral on the right-hand side of (9.2) is convergent for $\Re s>1+\sigma_0$, in which it serves as a valid definition of (the continuation of) $G_{\eps}^{\flat}(s)$:
\begin{equation}
G_{\eps}^{\flat}(s)=\frac{4}{\pi^2}\,\int_{\Re\mu=0} f^{\flat}(s-1+\eps\mu)\,\Phi(v,\,u;\,s-1+\eps\mu,\,\mu)\,d\mu.
\end{equation}
note that the convergence of this integral is locally uniform not only with respect to $s$, but with respect to $\Re s$ as well.
\end{proof}

\noindent
{\em Remarks\/} 9.1. (i) The analyticity of the function $G_{\eps}^{\flat}(s)$ in the domain $\{s\colon \Re (s-1)>\frac{\sigma_0}{2}\}$ will be used in the developments of Cauchy type to follow: it depended on the fact that the function $F_{\eps}^{\flat}$ is entire, or analytic in a sufficient domain. Again, this has been established in Theorem \ref{theo92} with the help of pseudodifferential arithmetic, and there would be no way to bypass these developments: complex analysis would not do here.\\

(ii) do not ascribe any importance to the factor 2 of the fraction $\frac{\sigma_0}{2}$ present here and there in what follows. It could have been replaced by $\kappa+1$ without difficulty, only replacing the constraint ``$Q$ is squarefree'' by the demand that, in the decomposition of $Q$, all prime numbers should appear with exponents $\leq \kappa$.\\

\begin{proposition}
Let $v,u$ satisfy the support conditions in Theorem {\em 8.2\/}, and let $\eps>0$. The function
\begin{equation}
\phi(s)= {\mathrm{Res}}_{\mu=-\hf} \Psi^{-,n}_{\eps}(s,\,\mu),
\end{equation}
with
\begin{align}
\Psi^{-,n}_{\eps}(s,\,\mu) &=\frac{4}{\pi^2}\,\,\times\,\frac{(1-2^{-s+1-\eps\mu})^{-1}}{\zeta(s-1+\eps\,\mu)}
\Phi^{-,n}(v,\,u;\,s-1+\eps\mu,\,\mu)\nonumber\\
&=\frac{4}{\pi^2}\,\,\times\,\frac{(1-2^{-s+1-\eps\mu})^{-1}}{\zeta(s-1+\eps\,\mu)} \int_1^{\infty} \overline{v}(t+t^{-1})\,\,t^{s-1+\eps\mu}\,u_-^{\mu,n}(t-t^{-1})\,\frac{dt}{t},
\end{align}
is meromorphic in the half-plane $\Re(s-1-\frac{\eps}{2})>\frac{\sigma_0}{2}$, with possible poles only at points $s$ such that $\zeta(s-1-\frac{\eps}{2})=0$.\\

Given such a point $s_0$, the function $\phi(s)$ will have a pole at $s=s_0$ if the pair $v,u$ satisfies the condition
$\Lan {\mathfrak E}_{1-s_0+\frac{\eps}{2}},\,\Wig(v,\,\delta^{n-1}u)\Ran \neq 0$.\\
\end{proposition}

\begin{proof}
According to Lemma 5.2 and Remark 5.1, the function $\phi$ cannot have a pole at $s$ unless $\zeta(s-1-\frac{\eps}{2})=0$. To analyze the Laurent expansion at $\mu=-\hf$ of the function
$\Psi^{-,n}_{\eps}(s,\,\mu)$, we must first manage so as to let $\mu$ show up in the last factor of the integrand only, using the (convergent) Taylor expansion
\begin{equation}
t^{s-1+\eps\mu}=t^{s-1-\frac{\eps}{2}}\sum_{j\geq 0}
\frac{(\eps(\mu+\hf))^j}{j\,!}\,(\log t)^j.
\end{equation}
With
\begin{equation}
w_j(x)=v(x)\l(\log \frac{x+\sqrt{x^2-4}}{2}\r)^j,
\end{equation}
one obtains
\begin{multline}
\Phi^{-,n}(v,\,u;\,s-1+\eps\mu,\,\mu)\\
=\sum_{j\geq 0} \frac{(\eps(\mu+\hf))^j}{j\,!} \int_0^{\infty}\overline{w_j}(t+t^{-1})\,t^{s-1-\frac{\eps}{2}}\,u_-^{\mu,n}(t-t^{-1})\,\frac{dt}{t}.
\end{multline}\\

Indeed, in terms of $x=t+t^{-1}$ with $0<t-t^{-1}<1$, one has $t=\frac{x+\sqrt{x^2-4}}{2}$ and, in (9.9), one can replace the factor $t^{s-1+\eps\mu}=t^{s-1-\frac{\eps}{2}}\,\exp((\eps(\mu+\hf)\,\log t))$ by the right-hand side of (9.10). One write $(\log t)^j=\l(\log\,\frac{x+\sqrt{x^2-4}}{2}\r)^j$, and one lets this factor be absorbed by the function $w_j$ which replaces $v$. Recall (8.34):
\begin{equation}
u_-^{-,n}(t-t^{-1})=\frac{1}{2\pi}\sum_{k=0}^{n-1} \frac{(-1)^k}{(\mu+\hf)^{k+1}}\,(\delta^ku)(t-t^{-1})+{\mathrm{O}}(1),\qquad \mu\to -\hf.
\end{equation}
Combining (9.12) and (9.13), one obtains
\begin{multline}
\Psi^{-,n}(s,\,\mu)\sim\frac{2}{\pi^3}\,\frac{1-2^{-s+1-\eps\mu}}{\zeta(s-1+\eps\mu)}\,
\sum_{j,k=0}^{n-1}(-1)^k\,\,\frac{\eps^j}{j\,!}\,\,(\mu+\hf)^{j-k-1}\\
\int_0^{\infty}\overline{w_j}(t+t^{-1})\,\,t^{s-1-\frac{\eps}{2}}\,\,(\log t)^j\,\,(\delta^ku)(t-t^{-1})\,\,\frac{dt}{t}.
\end{multline}
where the sign $\sim$ indicates that we have dropped the terms originating from the remainder term ${\mathrm O}(1)$ on the right-hand side of (9.13), which cannot contribute to a singularity at $\mu=-\hf$.\\

This is a linear combination of terms of the kind $(\mu+\hf)^j(\mu+\hf)^{-k-1}$ with $j\geq 0$ and $0\leq k \leq n-1$, the coefficients of which are products of the function
\begin{equation}
g(s-1+\eps\mu)=\frac{2}{\pi^3}\,\frac{1-2^{-s+1-\eps\mu}}{\zeta(s-1+\eps\mu)}
\end{equation}
by coefficients independent of $\mu$. Using for $\zeta(s-1-\frac{\eps}{2})\neq 0$ the Taylor expansion
\begin{equation}
g(s-1+\eps\mu)\sim \sum_{\ell\geq 0}
\frac{1}{\ell\,!}\,\,(\eps(\mu+\hf))^{\ell}\,\,g^{(\ell)}(s-1-\frac{\eps}{2}),
\end{equation}
one obtains finally, for the function $\Psi_{\eps}^{-,N}(s,\,\mu)$, a Laurent expansion which is a linear combination of powers $(\mu+\hf)^j(\mu+\hf)^{-k-1}(\mu+\hf)^{\ell}=(\mu+\hf)^{j-k-1+\ell}$ with coefficients independent of $\mu$.\\

Only the terms of (9.14) for which $j-k-1+\ell=-1$ contribute to the residue at $\mu=-\hf$: this leads to
\begin{multline}
{\mathrm{Res}}_{\mu=-\hf} \Psi^{-,n}(s,\,\mu)=\sum_{\ell \geq 0} \frac{1}{\ell\,!}\,\,
g^{(\ell)}(s-1-\frac{\eps}{2})
\sum_{\begin{array}{c}0\leq j,k \leq n-1\\ j-k=-\ell\end{array}} \\ (-1)^k\,\frac{\eps^{k}}{j\,!\,(k-j)\,!}\,
\int_0^{\infty}\overline{v}(t+t^{-1})\,\,t^{s-1-\frac{\eps}{2}}\,\,(\log t)^j\,\,(\delta^ku)(t-t^{-1})\,\,\frac{dt}{t}.
\end{multline}\\

The worst possible singularity as $s$ goes to $s_0$ such that $\zeta(s_0-1-\frac{\eps}{2})=0$ is obtained when taking $\ell$ as large as possible. This corresponds to the choice $j=0,\,k=n-1$, provided that the corresponding integral
\begin{equation}
\int_0^{\infty}\overline{v}(t+t^{-1})\,\,t^{s_0-1-\frac{\eps}{2}}\,\,
(\delta^{n-1}u)(t-t^{-1})\,\,\frac{dt}{t}=\Lan {\mathfrak E}_{1-s_0+\frac{\eps}{2}},\,\Wig(v,\,\delta^{n-1}u)\Ran
\end{equation}
is not zero.\\
\end{proof}

That, as will be needed later, given $n=1,2,\dots$ and $\nu\in \C,\,\nu\neq \pm 1$, there exist $v,u\in C^{\infty}(\R)$ with $v$ supported in $[2,\sqrt{8}]$ and $u$ in $[0,1]$ such that
\begin{equation}
<{\mathfrak E}_{-\nu},\,\Wig(v,\,\delta^{N-1}u>=\int_0^{\infty} \overline{v}(t+t^{-1})\,t^{\nu-1}\,(\delta^{N-1}u)(t-t^{-1})\,dt\neq 0
\end{equation}
goes as follows. Take $u$ real, $u\neq 0$. Then $\delta^{N-1}u\neq 0$ too because of the support condition on $u$ and, changing if necessary $u$ to $-u$, $(\delta^{N-1}u)(y)>0$ in some open subset $\omega$ of $]0,1[$. One has
$\int_0^{\infty}\overline{v_0}(t+t^{-1})\,(\delta^{N-1}u)(t-t^{-1})\,dt>0$ if $v_0\neq 0$ is real, nonnegative and supported in $\{\sqrt{y^2+4}\colon y\in\omega\}$, and it suffices to manage so that $\overline{v_0}(t+t^{-1})=
\overline{v}(t+t^{-1})\,t^{\nu-1}$, or $\overline{v_0}(x)=\overline{v}(x)\,\l(\frac{x+\sqrt{x^2-4}}{2}\r)^{\nu-1}$.\\

In the sequel, it will be necessary to decompose $\Phi(v,\,u;\,\nu,\,\mu)$
as $\Phi^{+,n}(v,\,u;\,\nu,\,\mu)+\Phi^{-,n}(v,\,u;\,\nu,\,\mu)$, with
\begin{equation}\label{713}
\Phi^{\pm,n}(v,\,u;\,\nu,\,\mu)=\int_0^{\infty} \overline{v}(t+t^{-1})\,t^{\nu}\,u_{\pm}^{\mu,n}(t-t^{-1})\,\frac{dt}{t}.
\end{equation}
Then, using (\ref{105}), one has the decomposition $G_{\eps}^{\flat}(v,\,u;\,s)=G_{\eps}^{\flat,+,n}(v,\,u;\,s)+G_{\eps}^{\flat,-,n}(v,\,u;\,s)$, with
\begin{multline}\label{107}
G_{\eps}^{\flat,\pm,n}(v,\,u;\,s)\\
=\frac{4}{\pi^2}\,\int_{\Re \mu=0}f^{\flat}(s-1+\eps\mu)\,
\Phi^{\pm,n}(v,\,u;\,s-1+\eps\mu,\,\mu),d\mu,
\end{multline}
a definition valid for $\Re(s-1)>\sigma_0$, while the sum extends analytically for $\Re(s-1)>\frac{\sigma_0}{2}$.\\

In view of their expressions in (\ref{928}), one has
\begin{align}\label{L1013}
|u^{\mu,n}_+(y)|&\leq \frac{1}{2\pi}\,\frac{1}{|\mu+\hf|^n}\,\,\frac{{\mathrm{sup}}\,|\delta^n u|}{\Re(\mu+\hf)}\quad {\mathrm{if}}\,\,\Re\mu>-\hf,\nonumber\\
|u^{\mu,n}_-(y)|&\leq\frac{1}{2\pi}\,\frac{1}{|\mu+\hf|^n}\,\,
\frac{{\mathrm{sup}}\,|\delta^n u|}{|\Re(\mu+\hf)|}\quad {\mathrm{if}}\,\,\Re\mu<-\hf.
\end{align}\\

In the equation (\ref{713}), one has when the integrand is nonzero $2<t+t^{-1}<\sqrt{8}$, so that
$0<t-t^{-1}<2$ and $1<t<\sqrt{2}+1$. Then, if $\Re\nu>0$, one has $|t^{\nu}|<(\sqrt{2}+1)^{\Re\nu}$. As $\int_1^{\sqrt{2}+1}|\overline{v}(t+t^{-1})|\,\frac{dt}{t}\leq
{\mathrm{sup}}\,|v|\,\,\times\,\log(\sqrt{2}+1)$, one obtains the bound
\begin{equation}\label{109}
|\Phi^{\pm,n}(v,\,u;\,\nu,\,\mu)| \leq C\,\, \,\frac{(\sqrt{2}+1)^{|\Re\nu|}}{|\mu+\hf|^n}\,\,
\frac{{\mathrm{sup}}\,|\delta^n u|}{|\Re(\mu+\hf)|}\,\times\,{\mathrm{sup}}\,|v|
\end{equation}
with $C=\frac{1}{2\pi}\,\log(\sqrt{2}+1)$. \\

\begin{theorem}\label{theoRH}(Riemann hypothesis)
The zeta function has no zero with a real part $>\hf$.\\
\end{theorem}

\begin{proof}
Recall (9.6). We write $G^{\flat}_{\eps}(v,\,u;\,s)=G^{\flat,+}_{\eps}(v,\,u;\,s)+G^{\flat,-}_{\eps}(v,\,u;\,s)$, with, for $\Re s>1+\sigma_0$ ($\sigma_0\geq\hf$),
\begin{equation}\label{Gflpm}
G^{\flat,\pm,n}_{\eps}(v,\,u;\,s)=\frac{4}{\pi^2}\,\int_{\Re\mu=0}
\frac{\l(1-2^{-s+1-\eps\mu}\r)^{-1}}{\zeta(s-1+\eps\mu)}\,
\Phi^{\pm,n}(v,\,u;\,s-1+\eps\mu,\,\mu)\,d\mu.
\end{equation}
Let $\rho$ be a zero of zeta such that $\Re\rho>\hf$. Recalling from Theorem \ref{theo102} that $G_{\eps}^{\flat}(v,\,u;\,s)$ is analytic for $\Re(s-1)>\frac{\sigma_0}{2}$, we shall obtain a contradiction from a proof that $G_{\eps}^{\flat,+,n}(v,\,u;\,s)$ is analytic for $\Re(s-1)>\frac{\sigma_0}{2}$ while, for some choice of the pair $v,u$,  $G_{\eps}^{\flat,-,n}(v,\,u;\,s)$ has a pole at $s_0=1+\rho+\frac{\eps}{2}$, a number with a real part $>1+\frac{\sigma_0}{2}$.\\

We now analyze separately the continuations to the half-plane $\{s\colon \Re(s-1)>\frac{\sigma_0}{2}\}$ of the functions $G_{\eps}^{\flat,+,n}(v,\,u;\,s)$ and $G_{\eps}^{\flat,-,n}(v,\,u;\,s)$. Considering the first, we prepare it by a change of contour, replacing the line $\Re\mu=0$ by the line $\Re\mu=a$, for some $a>\eps^{-1}(1-\frac{\sigma_0}{2})$. We do this still under the assumption that $\Re(s-1)>\sigma_0$, for which the $d\mu$-integrals on the lines $\Re\mu=0$ and $\Re\mu=a$ are convergent in view of (\ref{109}): as in the proof of Theorem 9.1, this convergence is locally uniform not only with respect to $s$, but with respect to $\Re s$ as well. From the same inequality, the integrals on the two horizontal lines with large imaginary parts we have in mind when approaching the strip $\{\mu\colon 0\leq \Re\mu\leq a\}$ by a large rectangle go to zero as the heights of the horizontal sides go to $\pm \infty$. After this change of contour has been made, the new integral is an analytic function of $s$ for $\Re(s-1)>\frac{\sigma_0}{2}$.\\

Consider now the integral (9.24) for the function $G_{\varepsilon}^{\flat,-,n}(v,\,u;\,s)$.
Recall that $\zeta(\rho)=0,\,\Re\rho>\hf$. Fix $\frac{\eps}{2}=\Re(\sigma_0+\rho-1)=(\sigma_0-\hf)+\Re(\rho-\hf)$, a positive number.
With $s_0=\rho+1+\frac{\eps}{2}$, the integral
\begin{equation}
\frac{4}{\pi^2} \int_{\Re\mu=0} \frac{\l(1-2^{-s_0+1-\eps\mu}\r)^{-1}}{\zeta(s_0-1+\eps\mu)}\,\Phi^{-,n}(v,\,u;\,s_0-1+\eps\mu,\,\mu)\,d\mu
\end{equation}
is convergent since $s_0$ satisfies $\Re(s_0-1)=\Re(\sigma_0+2\rho-1)=\sigma'_0>\sigma_0$, and one can build a simple path $\beta$ from the point $\mu=-\hf$ to some point of the line $\Re\mu=0$ with the following properties: $-\hf\leq \Re\mu\leq 0$ on this path, and its image $\widetilde{\beta}$ under the map $\mu\mapsto \nu=s_0-1+\eps\mu$ intersects the set of zeros of zeta at the point $\rho=s_0-1-\frac{\eps}{2}$ only. Constructing $\widetilde{\beta}$ if one drops the condition $\Re\mu\geq -\hf$ on the path is a connectedness argument; after that, benefitting from a bound of $|\Im\mu|$, there are only finitely many zeros to be avoided, an immediate task.
If $k=1,2,\dots$ is the order of the zero $\rho$ of zeta, the function $\nu\mapsto \frac{\zeta(\nu)}{(\nu-\rho)^k}$ is continuous and nonzero for $\nu$ on the path $\widetilde{\beta}$. There exists $c>0$ such that $\l|\frac{\zeta(\nu)}{(\nu-\rho)^k}\r|\geq c$ for $\nu\in \widetilde{\beta}$. Then, changing $c$ to $\frac{c}{2}$, the similar inequality is valid for $\nu$ in some neighborhood of $\widetilde{\beta}$. It is also valid, of course, for $\Re\nu=\sigma'_0$, hence for $\nu$ in the closure of some domain bounded by a line $\widetilde{\gamma}$, a $C^{\infty}$ deformation of the line $\Re\nu=\sigma'_0$ enclosing the point $\rho$. In particular, $\zeta(\nu)\neq 0$ for $\nu \in \widetilde{\gamma}$.\\

The line $\widetilde{\gamma}$ is the image under the map $\mu\mapsto \nu=s_0-1+\eps\mu$ of a contour $\gamma$, a deformation of the line $\Re\mu=0$ such that, for $\mu$ in the domain delimited by $\gamma$ and the line $\Re\mu=0$, one has $\zeta(s_0-1+\eps\mu)\neq 0$ if $\mu\neq -\hf$. One replaces in (9.23) the line $\Re\mu=0$ by $\gamma$, adding $2i\pi$ times the residue at $\mu=-\hf$: by Proposition 9.2, this residue is, for some choice of the pair $v,u$, singular at $s_0$ and at no other point $s=1+\rho-\eps\mu$ lying in the open set with boundary $\widetilde{\gamma}$. This time, the construction and estimates are not locally uniform with respect to $\Re s_0$, but they are so with respect to $s_0$, which suffices for analyticity.\\

This concludes this proof of the Riemann hypothesis.\\
\end{proof}

\end{document}